\documentclass[11pt]{amsart}
\usepackage{graphicx}
\usepackage{amssymb}
\usepackage{amsmath}
\usepackage{amsthm,amsfonts,bbm}
\usepackage{amscd}
\usepackage{geometry}
\usepackage[all,2cell]{xy}
\usepackage{epsfig,epstopdf}
\usepackage{mathrsfs}
\usepackage{hyperref}
\usepackage{cite}
\usepackage{color}

\UseAllTwocells \SilentMatrices

\newtheorem{thm}{Theorem}[section]
\newtheorem{cor}[thm]{Corollary}
\newtheorem{lem}[thm]{Lemma}
\newtheorem{pro}[thm]{Problem}
\newtheorem{prop}[thm]{Proposition}
\newtheorem{myCli}{Claim}
\theoremstyle{definition}
\newtheorem{defi}[thm]{Definition}
\theoremstyle{remark}

\newtheorem{conj}[thm]{\bf Conjecture}

\numberwithin{equation}{section}
\numberwithin{figure}{section}
\begin{document}
\title[Spectral Tur\'an-type  problems  for the $\alpha$-spectral radius]{Spectral Tur\'an-type  problems  for the $\alpha$-spectral radius of hypergraphs with degree stability}

\author[J. Zheng]{Jian Zheng}
\address{School of  Mathematics and Statistics, Jiangxi Normal University, Nanchang 330022,  China}
\email{zhengj@jxnu.edu.cn}

\author[H. Li]{Honghai Li$^\dag$}
\address{School of  Mathematics and Statistics, Jiangxi Normal University, Nanchang 330022,  China}
\email{lhh@jxnu.edu.cn}
\thanks{$^\dag$The corresponding author. H. Li was supported by National Natural Science Foundation of China (No. 12161047) and   Jiangxi Provincial Natural Science Foundation (No. 20224BCD41001).}

\author[L. Su]{Li Su$^\ddag$}
\address{School of  Mathematics and Statistics, Jiangxi Normal University, Nanchang 330022,  China}
\email{suli@jxnu.edu.cn}
\thanks{$^\ddag$L. Su was supported by National Natural Science Foundation of China (No. 12061038).}

%\subjclass[2000]{Primary 05C65, 15A69; Secondary 13P15, 14M99}

%\newline $^{\dag}$ Corresponding author.
%\newline  E-mail addresses: zhengj@jxnu.edu.cn, lhh@jxnu.edu.cn
%\newline School of  Mathematics and Statistics, Jiangxi Normal University, Nanchang,  Jiangxi 330022,  China.}

\begin{abstract}
 An $r$-pattern $P$ is  an ordered pair $P=([l],E)$, where $l$ is a positive integer and $E$ is a set of $r$-multisets with elements from $[l]$. An $r$-graph $H$ is said to be  $P$-colorable if there is a homomorphism  $\phi$: $V(H)\rightarrow [l]$ such that  $\{\phi(v_{1}),\ldots,\phi(v_{r})\}\in E$ for every edge $\{v_{1},\ldots,v_{r}\}\in E(H)$. Let $\mathrm{Col}(P)$ denote the family of all $P$-colorable $r$-graphs. This paper studies spectral extremal  problems  for $\alpha$-spectral radius of hypergraphs via  analytic techniques.
 We first prove that for any $r$-pattern $P$, the hypergraph attaining the maximum $\alpha$-spectral radius in $\mathrm{Col}(P)$ is asymptotically regular. Specifically, we establish asymptotically tight lower bounds for the minimum component of the principal eigenvector and the minimum degree of the spectral extremal hypergraphs in  $\mathrm{Col}(P)$.
 Building on this regularity, we further  show that for any family $\mathcal{F}$ of $r$-graphs that is degree-stable with respect to $\mathrm{Col}(P)$, spectral Tur\'an-type  problems can be completely reduced to spectral extremal problems within $\mathrm{Col}(P)$. As an application, we   determine  the  maximum $\alpha$-spectral radius ($\alpha\geq1$) among all $n$-vertex $F^{(r)}$-free $r$-graphs, where $F^{(r)}$ is the $r$-expansion of the color-critical graph $F$. This provides a powerful reduction tool for handling  spectral Tur\'{a}n-type problems in hypergraphs. Finally, leveraging the spectral method, we derive a corresponding edge Tur\'an extremal result. More precisely, we show that if $\mathcal{F}$ is degree-stable with respect to $\mathrm{Col}(P)$, then every $\mathcal{F}$-free edge extremal hypergraph must be a
 $P$-colorable hypergraph.
\end{abstract}

\maketitle

\section{Introduction}

A  \emph{hypergraph} $H$ consists of a set of vertices   $V(H)$  and a set of  edges     $E(H)\subseteq 2^{V(H)}$,  where $2^{V(H)}$ is the power set of
$V(H)$.  The \emph{order} and  \emph{size} of $H$ are defined as  $\nu(H):=|V(H)|$ and  $e(H):=|E(H)|$, respectively.   If every edge of $H$ contains exactly $r$ vertices, then $H$ is called an \emph{$r$-uniform} hypergraph (\emph{$r$}-graph, for short). An ordinary graph is precisely a  $2$-uniform hypergraph. For two $r$-graphs $G$ and $H$, we say  $G$ is  a \emph{subgraph} of $H$ if $V(G)\subseteq V(H)$ and $E(G)\subseteq E(H)$. The subgraph $G$  is  \emph{induced} by $V(G)$  if $E(G)=\{e\in E(H):e\subseteq V(G)\}$. For any vertex $v\in V(H)$, we denote by $H-v$ the subgraph of $H$  induced by $V(H)\backslash\{v\}$.

 Let $\mathcal{F}$ be a family of $r$-graphs. An $r$-graph $G$ is called \emph{$\mathcal{F}$-free} if no member $F\in \mathcal{F}$ is a subgraph of $G$. The \emph{Tur\'an number} $ex(n, \mathcal{F})$ is the maximum size of an $\mathcal{F}$-free $r$-graph on $n$ vertices. The classical Tur\'an's Theorem \cite{T1941} states that  the  balanced complete $l$-partite graph $T_{l}(n)$ (known as the Tur\'an graph) is the unique extremal graph attaining the maximum size among all $n$-vertex $K_{l+1}$-free graphs.
 Tur\'an problems constitute a central theme in extremal combinatorics; we refer to  \cite{K2011} for a comprehensive survey.  In  $2007$,  Nikiforov \cite{N2007} proved a spectral analogue of Tur\'an's Theorem, showing that the Tur\'an graph is also the unique spectral extremal graph. Since then, the  spectral Tur\'an-type
  problems have attracted great attention; see  \cite{LiLF} for  recent  developments.

In \cite{N2014B,N2014A}, Nikiforov conducted a systematic study of the  $\alpha$-spectral radius  of hypergraphs using analytical methods.  The  $\alpha$-spectral radius of an $r$-graph $G$ (see  Section 2 for its definition), denoted by $\lambda^{(\alpha)}(G)$, unifies several fundamental concepts: the largest $Z$-eigenvalue, the largest  $H$-eigenvalue (see \cite{Q2005}), the Lagrangian, and the size of $G$. Consequently, it provides a  unified framework for addressing associated extremal problems. This parameter  was first introduced by Keevash, Lenz and Mubayi\cite{KLM2014}, who  established  two key criteria for spectral extremal problems.
Applying one of these criteria, they determined the maximum $\alpha$-spectral radius  among all $n$-vertex $3$-graphs that avoid the Fano plane. Nikiforov\cite{N2014B} further investigated several extremal problems of the following type:
\begin{pro}
  What is the maximum value of $\lambda^{(\alpha)}(G)$ over all $r$-graphs $G$ satisfying a given property $\mathcal{P}$?
\end{pro}
 Erd\H{o}s, Ko and Rado proved that if $G$ is a $t$-intersecting  $r$-graph on $n$ vertices, then $e(G)\leq\binom{n-t}{r-t}$, with equality holding if and only if  $G=S_{t, n}^r$, where $S_{t, n}^r$ denotes the complete $t$-star of order $n$. They further established a  stability result: there exists a constant $c=c(r, t)>0$ such that if $e(G)>c\binom{n-t}{r-t-1}$, then $G\subseteq S_{t, n}^r$. Building on this stability result, Nikiforov\cite{N2014B} demonstrated that for  sufficiently large $n$,  $S_{t, n}^r$ is the unique extremal hypergraph achieving  the maximum $\alpha$-spectral radius among all $t$-intersecting $n$-vertex $r$-graphs.  Although Keevash, Lenz  and Mubayi \cite{KLM2014}  had previously obtained this result using another criterion, their approach involved a more complex framework. These foundational contributions have significantly advanced the study of extremal problems for the $\alpha$-spectral radius. Subsequently,  Kang, Nikiforov and Yuan \cite{KNY2014}  resolved several extremal problems concerning the $\alpha$-spectral radius of $n$-vertex $k$-partite and $k$-chromatic
hypergraphs,  generalizing earlier work by Mubayi and
Talbot \cite{MT2008}.
Ni, Liu and Kang\cite{NLK2022} determined  the maximum $\alpha$-spectral radius of all cancellative $3$-graphs.
Furthermore, Kang, Liu, Lu and Wang\cite{KLLW2021} investigated the $\alpha$-spectral radius  of Berge-$G$
 hypergraphs. In particular, for $\alpha\geq1$, they determined  the 3-graphs that attain the
maximum $\alpha$-spectral  radius among all Berge-$G$ hypergraphs when $G$ is a path, a cycle or a star.

 Let $H$ be an ordinary graph (i.e., a $2$-graph). The $r$-\emph{expansion} $H^{(r)}$ of $H$ is the $r$-graph obtained from $H$ by replacing each edge with $r-2$ new vertices, where the added vertices are disjoint from $V(H)$ and distinct edges receive disjoint sets of new vertices.
 Let $F_{r,l}$ be the $r$-graph with edge set
\[
E(F_{r,l}) = \big\{[r]\big\} \cup \bigcup_{\{i,j\}\in\binom{[l]}{2}\setminus\binom{[r]}{2}} \{E_{ij}\cup\{i,j\}\},
\]
where the $E_{ij}$ are pairwise disjoint $(r-2)$-sets disjoint from $[l]$. Generalizing the Tur\'an's theorem
 to hypergraphs,  Pikhurko\cite{P2013}  determined the exact
 Tur\'an number of $K_{l+1}^{(r)}$ , while
  Mubayi and Pikhurko\cite{MO2006} established the Tur\'an number of  $F_{r,l+1}$.
For  sufficiently large $n$ and $\alpha\geq r\geq 2$,  Liu et al. \cite{LNWK2024} proved that  the balanced complete $l$-partite $r$-graph $T_l^r(n)$ on $n$
 vertices   is the unique hypergraph  maximizing
 the $\alpha$-spectral radius among all $n$-vertex $K_{l+1}^{(r)}$-free (or $F_{r,l+1}$-free) $r$-graphs. This result generalizes the findings of  Pikhurko and Mubayi.

   Note that for an $r$-graph $G$, the quantity $\lambda^{(r)}(G)/(r-1)!$ is exactly its spectral radius, i.e., the maximum modulus of all eigenvalues of the adjacency tensor of $G$. Thanks to the development of the spectral theory of nonnegative tensors,  the spectral radius of hypergraphs is now relatively well understood, and a large body of research has been devoted to extremal problems for the spectral radius of hypergraphs.
  Gao, Chang, and Hou~\cite{GCH2022} studied the spectral extremal problem for linear $K_{r+1}^{(r)}$-free $r$-graphs. They proved that for sufficiently large $n$, the spectral radius of an $n$-vertex $K_{r+1}^{(r)}$-free linear $r$-graph is at most $n/r$.
  She, Fan, Kang and Hou\cite{SFKH2023} generalized this result by establishing a sharp upper bound on the spectral radius of $F^{(r)}$-free linear $r$-graphs,  where $F$ is a color-critical graph with chromatic number  $k\geq r+1$.  Ellingham, Lu  and Wang \cite{ELW2022} characterized the extremal hypergraph attaining  maximum  spectral radius among all outerplanar $3$-graphs on $n$ vertices  via its $2$-shadow graph.  There  exist numerous results on extremal characterizations of the spectral radius of hypergraphs;  for more details, we refer the reader to  \cite{BIJ2017,CDS2024,FHY2023,KZS2021, KLLW2021, LZ2017,OQY2017, SL2024, WY2024, XHW2023, YYSWZ2021, ZLG2020,ZZ2025}.

Recently, Zheng, Li  and Fan~\cite{ZLF2024} established a spectral stability result for nondegenerate hypergraphs,  generalizing the Keevash--Lenz--Mubayi criterion \cite[Theorem $1.4$]{KLM2014}. Their main result can be stated as follows:

\begin{thm}[\hspace{1sp}\cite{ZLF2024}]\label{oldcri}
Let $r\geq2$, $\alpha>1$, $0<\varepsilon<1$, and $\mathcal{F}$ be a family of $r$-graphs with $\pi(\mathcal{F})>0$. Let $\mathcal{G}_{n}$ be the collection of all $n$-vertex $\mathcal{F}$-free $r$-graphs with minimum degree at least  $(1-\varepsilon)\pi(\mathcal{F})\binom{n}{r-1}$, and let
$\lambda^{(\alpha)}(\mathcal{G}_{n})=\max\{\lambda^{(\alpha)}(G):G\in \mathcal{G}_{n}\}$.
Suppose that
there exists  $N$ such that
 for all $n\geq N$,
\begin{equation}\label{oldcri1}
\lambda^{(\alpha)}(\mathcal{G}_{n}) \ge  \lambda^{(\alpha)}(\mathcal{G}_{n-1}) + (r-r/\alpha)(1-\varepsilon')\pi(\mathcal{F})n^{r-r/\alpha-1},
\end{equation}
where $\varepsilon'=\varepsilon\pi(\mathcal{F})(\alpha-1)/(2r \alpha)$.
Then there exists $n_{0}$ such that if $H$ is an $\mathcal{F}$-free $r$-graph on $n\geq n_{0}$ vertices, then
$$\lambda^{(\alpha)}(H)\leq \lambda^{(\alpha)}(\mathcal{G}_{n}).$$
In addition, if the equality holds, then $H\in\mathcal{G}_{n}$.
\end{thm}
By applying this Theorem \ref{oldcri}, Zheng et al. determined the maximum $\alpha$-spectral radius for some classes of hypergraphs and characterized the corresponding extremal hypergraphs (see Section $4$ in \cite{ZLF2024}),  which demonstrates the power of the  spectral stability result. However,  the condition (\ref{oldcri1}) is generally difficult to verify. Motivated by this challenge, we seek to develop a simpler and more feasible method to address spectral Tur\'an-type problems. The necessary definitions are presented below.

Let  $G$ be an $r$-graph on $n$ vertices.  Assume that $V(G)$  has a partition $V(G)=V_{1}\cup\ldots\cup V_{l}$, and
 this partition is called \emph{balanced} if $V_{1},\ldots,V_{l}$ have almost equal sizes, namely $\lfloor n/l\rfloor \leq|V_{i}|\leq\lceil n/l\rceil$ for all $i \in [l]$.
Then  $G$ is called \emph{$l$-partite} (resp. \emph{$l$-chromatic}) if for any edge $e\in E(G)$, $|e \cap V_{i}|\leq1$ (resp.
 $|e \cap V_{i}|\leq r-1$)
 for all $i\in [l]$;  furthermore, an $l$-partite (resp. $l$-chromatic) $r$-graph $G$ is called \emph{complete} if it is edge-maximal with respect to the given partition constraints (i.e.  $\forall e\in \binom{[n]}{r}$, $e\in E(G)$ whenever $|e \cap V_{i}|\leq1$ (resp. $|e \cap V_{i}|\leq r-1$)  for $i\in [l]$). Note that if an $l$-partite $r$-graph has at least one edge, then $l\geq r$.  We denote by $T^{r}_{l}(n)$ (resp.  $Q^{r}_{l}(n)$)
the balanced complete $l$-partite (resp. $l$-chromatic) $r$-graph on $n$ vertices. Clearly, $T^{2}_{l}(n)$ and $Q^{2}_{l}(n)$ coincide, and both are isomorphic to Tur\'an graph $T_{l}(n)$.

An \emph{$r$-multiset} is a  multiset of cardinality $r$. An \emph{$r$-pattern} is an ordered pair $P=([l],E)$ where $l$ is a positive integer and $E$ is a set of $r$-multisets with elements from $[l]$. Clearly,  the $r$-pattern is a generalization of the $r$-graph. For convenience, we  use $E$ (or an $r$-graph $G$ with $E(G)=E$)  to represent the pattern $P$ when the context is clear. Given an $r$-graph $H$ and an $r$-pattern $P=([l],E)$, a map $\phi$: $V(H)\rightarrow [l]$ is a \emph{homomorphism} from $H$ to $P$ if the $r$-multiset $\{\phi(v_{1}),\ldots,\phi(v_{r})\}$ belongs to $E$ for every edge $\{v_{1},\ldots,v_{r}\}\in E(H)$. We say $H$ is \emph{$P$-colorable} if there is a homomorphism from $H$ to $P$. For example, any $l$-partite $r$-graph is $K^{r}_{l}$-colorable, where $K^{r}_{l}$ is the complete $r$-graph on $l$ vertices. Denote by $\mathrm{Col}(P)$   the family of all $P$-colorable $r$-graphs.
The notion of pattern was introduced by Pikhurko \cite{PIK2014} in the study of  hypergraph Tur\'an density and has also been utilized in \cite{HLL2023,LP2023}.

This paper develops a unified spectral framework for extremal problems concerning the $\alpha$-spectral radius of hypergraphs. We begin with a systematic review of the necessary terminology, notation, and related results.  In Section 3,   we   prove that  for any $r$-pattern $P$, a hypergraph attaining the maximum $\alpha$-spectral radius within $\mathrm{Col}(P)$  is asymptotically regular (see Theorem~\ref{Co}).  An important consequence of this regularity is that the vertex partition of the spectral extremal   $k$-chromatic hypergraph is approximately balanced (see Theorem~\ref{chro}). This result provides  partial support for a  conjecture of Kang, Nikiforov, and Yuan regarding the structure of the spectral extremal $k$-chromatic hypergraph  \cite{KNY2014}.
In Section 4, we establish a general reduction theorem: if a family $\mathcal{F}$ of $r$-graphs  is degree-stable with respect to  $\mathrm{Col}(P)$, then every $\mathcal{F}$-free $r$-graph attaining the maximum $\alpha$-spectral radius must be $P$-colorable (see Theorem~\ref{tp}).  A major application arises by specializing to the case $P = K_{l}^r$. For any family $\mathcal{F}$ that is degree-stable with respect to $\mathrm{Col}(K_{l}^r)$ and for all sufficiently large $n$, the balanced complete $l$-partite $r$-graph $T_l^r(n)$ achieves the maximum $\alpha$-spectral radius ($\alpha \geq 1$) among all $n$-vertex $\mathcal{F}$-free $r$-graphs. This directly resolves the spectral Tur\'{a}n-type problem for the $r$-expansion $F^{(r)}$ of any color-critical graph $F$ (see Theorem \ref{ccg}).
In the final section, we solve edge extremal problems using spectral methods, analyze  relations between  edge extremal hypergraphs and  spectral extremal hypergraphs, and pose an open  question.

% \subsection{Graph properties}

%Obviously, $Mon(\mathcal{F})\subseteq Her(\mathcal{F})$.

 %To date there are very few results on spectral Tur\'an problems of hypergraph.

 %In fact, we demonstrate
%that such a pair \( (\mathcal{F}, P) \) is a Tur\'an pair,  as per the definition in reference \cite{HLL2023}.
%\medskip{In this paper, we investigate the maximum $\alpha$-spectral radius of a family of hypergraphs with hereditary properties. In Section $2$, we introduce the basic definitions and present the necessary lemmas. In Section $3$, we firstly use Lemma \ref{Co} to show Theorem \ref{chro}. Second, we prove Theorem \ref{tp},  which extends the results in \cite[Theorem $1.2$]{HLZ2024}. Meanwhile, as an application of Theorem \ref{tp}, we present some results of the spectral Tur\'an-type extremal problems; see Corollary \ref{trl} and Corollary \ref{b4n}.  Finally, we obtain some results for edge Tur\'an-type extremal problems using spectral method; see Theorem \ref{colore} and Corollary \ref{ebound}. In Section $4$, we present two open problems that naturally arise from the discussions in this paper.}

\section{Preliminaries}
 A property of $r$-graphs is a family of $r$-graphs closed under isomorphisms.  For a property $\mathcal{P}$, denoted by $\mathcal{P}_{n}$  the collection of $r$-graphs in $\mathcal{P}$ of order $n$. A property is called \emph{monotone} if it is closed under taking subgraphs, and \emph{hereditary} if it is closed under taking induced subgraphs. Given a  family  $\mathcal{F}$ of $r$-graphs, the class of all $\mathcal{F}$-free $r$-graphs constitutes a monotone property, denoted by $Mon(\mathcal{F})$. Similarly, the class of all $r$-graphs that do not contain any $F\in \mathcal{F}$ as an induced subgraph  forms a hereditary property. Evidently, every monotone property is hereditary. Throughout our discussion,  we  assume that any hereditary property $\mathcal{P}$ of $r$-graphs satisfies the following conditions:
 \begin{itemize}
  \item [\rm (a)]
  There exists an $r$-graph $G\in\mathcal{P}$ that contains at least one edge.
    \item [\rm (b)]
    For any $H\in \mathcal{P}$, the disjoint union  of $H$ and an isolated vertex belongs to $\mathcal{P}$.
  \end{itemize}

A central topic in extremal hypergraph theory is the following: Given a hereditary property $\mathcal{P}$ of $r$-graphs, determine
$$ex(\mathcal{P},n):=\max_{G\in \mathcal{P}_{n}}e(G).$$
By the Katona-Nemetz-Simonovits averaging argument \cite{KNS1964}, the ratio $ex(\mathcal{P},n)/\binom{n}{r}$ is non-increasing in $n$.
Consequently, the limit
\begin{equation}\label{eltt}
\hat{\pi}(\mathcal{P}):=\lim\limits_{n\to \infty}\frac{ex(\mathcal{P},n)}{\binom{n}{r}}
\end{equation}
 exists and is called the \emph{edge density} of $\mathcal{P}$. Thus, determining the asymptotic behavior of $ex(\mathcal{P},n)$ is equivalent to  finding  $\hat{\pi}(\mathcal{P})$. When $\mathcal{P}=Mon(\mathcal{F})$ for
a family  $\mathcal{F}$ of $r$-graphs,  $ex(\mathcal{P},n)$ is known as the Tur\'an number of $\mathcal{F}$, and
$\hat{\pi}(\mathcal{P})$ corresponds to the well-studied \emph{Tur\'an density} $\pi(\mathcal{F})$ of $\mathcal{F}$.

For  an $r$-graph $G$ on $n$ vertices and $\alpha\geq1$, the \emph{Lagrangian polynomial} $P_{G}(\mathbf{x})$ of $G$ is defined
 as $$P_{G}(\mathbf{x})=r!\sum_{\{i_{1},\ldots,i_{r}\}\in E(G)}x_{i_{1}}\cdots x_{i_{r}},$$
 and the \emph{$\alpha$-spectral radius} of $G$ is defined as
$$\lambda^{(\alpha)}(G)=\max_{\|\mathbf{x}\|_{\alpha}=1}P_{G}(\mathbf{x}),$$
where $\mathbf{x}=(x_{1},\ldots,x_{n})\in \mathbb{R}^{n}$  and $\|\mathbf{x}\|_{\alpha}:=(|x_{1}|^{\alpha}+\cdots +|x_{n}|^{\alpha})^{1/\alpha}$. If $\mathbf{x}\in \mathbb{R}^{n}$ is a vector such that $\|\mathbf{x}\|_{\alpha}=1$ and
$\lambda^{(\alpha)}(G)=P_{G}(\mathbf{x})$, then $\mathbf{x}$ is called an \emph{eigenvector} of $G$ corresponding to $\lambda^{(\alpha)}(G)$. Clearly, there always exists a nonnegative eigenvector corresponding to $\lambda^{(\alpha)}(G)$, called  a \emph{principal eigenvector} of $G$.
Moreover, if a principal eigenvector $\mathbf{x}$ is strictly  positive (i.e., $x_{v}>0$ for all $v\in V(G)$), then we call it a \emph{Perron-Frobenius eigenvector} of $G$.

The classical Perron-Frobenius theory for \( \alpha = r = 2 \)  guarantees that the principal eigenvector \( \mathbf{x} \) of any connected \( 2 \)-graph \( H \) is also its Perron-Frobenius eigenvector. By extending this to  $r$-graphs $G$, Nikiforov \cite{N2014B} established a comprehensive Perron-Frobenius theory for \( \lambda^{(\alpha)}(G) \). Specifically, Nikiforov \cite[Theorem $5.10$]{N2014B} demonstrated that this holds true if \( \alpha\geq r \geq 2 \). However, for \( 1 < \alpha < r \), it is shown by \cite[Proposition $4.3$]{N2014B} that \( G \) may not have Perron-Frobenius eigenvector.
 Let $\mathbb{PE}_{\alpha}(G)$ denote the set of all principal eigenvectors for $\lambda^{(\alpha)}(G)$. Define
$$x_{\textup{min}}^{(\alpha)}(G):=\inf_{\mathbf{x} \in \mathbb{PE}_{\alpha}(G)}\mathbf{x}_{\textup{min}},$$
where $\mathbf{x}_{\textup{min}}$ represents the minimum component of the vector $\mathbf{x}$. We simply write $x_{\textup{min}}(G)$ instead of $x_{\textup{min}}^{(\alpha)}(G)$   when no confusion occurs.

For $\alpha>1$, the principal eigenvector $\mathbf{x}=(x_{1},\ldots,x_{n})$ of an
 $r$-graph $G$ satisfies the following equations derived from Lagrange's method:
\begin{equation}\label{eigen}
\lambda^{(\alpha)}(G)x_{i}^{\alpha-1}=(r-1)!\sum_{\{i,i_{2},\ldots,i_{r}\}\in E(G)}x_{i_{2}}\cdots x_{i_{r}},~ \mbox{for}~
1\leq i\leq n.
\end{equation}

\begin{lem}[\hspace{1sp}\cite{N2014A,N2014B}]\label{conti}
If $\alpha\geq1$ and $G$ is an $r$-graph, then $\lambda^{(\alpha)}(G)$ is an increasing and continuous function in $\alpha$.
Moreover,
$$\lim\limits_{\alpha\to \infty}\lambda^{(\alpha)}(G)=r!e(G).$$
\end{lem}

\medskip{For a hereditary property $\mathcal{P}$ of $r$-graphs, define}
$$\lambda^{(\alpha)}(\mathcal{P},n):=\max_{G\in \mathcal{P}_{n}}\lambda^{(\alpha)}(G).$$
By analogy with the edge  density in \eqref{eltt}, Nikiforov formulated the spectral analogue as follows:

\begin{lem}[\hspace{1sp}\cite{N2014A,N2014B}]\label{exi}
Let $\alpha\geq 1$. If $\mathcal{P}$ is a hereditary property of $r$-graphs, then the limit
$$\lambda^{(\alpha)}(\mathcal{P}):=\lim\limits_{n\to \infty}\lambda^{(\alpha)}(\mathcal{P},n)n^{r/\alpha-r}$$
exists. If $\alpha=1$, then $\lambda^{(1)}(\mathcal{P},n)$ is increasing, and so
$$\lambda^{(1)}(\mathcal{P},n)\leq \lambda^{(1)}(\mathcal{P}).$$
If $\alpha>1$, then $\lambda^{(\alpha)}(\mathcal{P})$ satisfies
$$\lambda^{(\alpha)}(\mathcal{P})\leq\frac{\lambda^{(\alpha)}(\mathcal{P},n)n^{r/\alpha}}{(n)_{r}},$$
where $(n)_{r}=n(n-1)\cdots (n-r+1)$.
\end{lem}

\noindent {\bf Remark.}
\begin{itemize}
\item
For any hereditary property  $\mathcal{P}$,  assumption $(a)$ implies that $\lambda^{(1)}(\mathcal{P})\geq\lambda^{(1)}(K^{r}_{r})>0$.
\item
Consider the family of $r$-graphs
$\mathcal{R}=\{K^{r}_{r}\}\cup \{G:E(G)=\emptyset\}$. Although
$\mathcal{R}$ is hereditary, for $n>r$, we have
$\lambda^{(1)}(\mathcal{R},r)=\lambda^{(1)}(K^{r}_{r})>\lambda^{(1)}(\mathcal{R},n)=0$.  Therefore, assumption $(b)$ ensures that $\lambda^{(1)}(\mathcal{P},n)$
is increasing in $n$.
\end{itemize}

\medskip{For $G\in \mathcal{P}_{n}$ with $e(G)=ex(\mathcal{P},n)$, the $n$-vector $\mathbf{x}=(n^{-1/\alpha},\ldots,n^{-1/\alpha})$ yields}
\begin{equation}\label{zhy}
\lambda^{(\alpha)}(G)\geq P_{G}(\mathbf{x})=r!e(G)/n^{r/\alpha}=r!ex(\mathcal{P},n)/n^{r/\alpha}.
\end{equation}
Thus
$$\lambda^{(\alpha)}(\mathcal{P},n)\geq \lambda^{(\alpha)}(G) \geq r!ex(\mathcal{P},n)/n^{r/\alpha},$$
which implies
$$\frac{\lambda^{(\alpha)}(\mathcal{P},n)n^{r/\alpha}}{(n)_{r}}\geq \frac{ex(\mathcal{P},n)}{\binom{n}{r}}.$$
Taking  $n\to \infty$ and applying Lemma \ref{exi},   we obtain for $\alpha\geq1$,
\begin{equation}\label{pp}
\lambda^{(\alpha)}(\mathcal{P})\geq\hat{\pi}(\mathcal{P}).
 \end{equation}

 The following theorem establishes that equality in (\ref{pp}) necessarily holds  for all $\alpha>1$, demonstrating an asymptotic equivalence between the extremal problem  $ex(\mathcal{P},n)$ and the spectral extremal problem  $\lambda^{(\alpha)}(\mathcal{P},n)$.

\begin{lem}[\hspace{1sp}\cite{N2014A,N2014B}]\label{t1}
If $\mathcal{P}$ is a hereditary property of $r$-graphs, then for every $\alpha>1$,
$$\lambda^{(\alpha)}(\mathcal{P})=\hat{\pi}(\mathcal{P}).$$
\end{lem}

We now extend the definition of graph blow-ups to hypergraphs.
\begin{defi}
Let $H$ be an $r$-graph on $n$ vertices and
 $\mathbf{t}=(k_{1},\ldots,k_{n})$ be a vector of positive integers. The blow-up of $H$ with respect to $\mathbf{t}$, denoted by $H(\mathbf{t})$, is the $r$-graph obtained by replacing each vertex
 $i\in V(H)$ with a  vertex class   $V_{i}$ of size $k_{i}$,  and if $\{i_{1},\ldots,i_{r}\}\in E(H)$, then $\{i_{1,j_{1}},\ldots,i_{r,j_{r}}\} \in E(H(\mathbf{t}))$ for every $i_{1,j_{1}}\in V_{i_{1}},\ldots,i_{r,j_{r}}\in V_{i_{r}}$.
 \end{defi}

Using the concept of blow-up, we define an important graph property.
 \begin{defi}[Multiplicative property]
 A graph property $\mathcal{M}$ is   \emph{multiplicative} if $G\in \mathcal{M}_{n}$  implies that $G(\mathbf{t})\in \mathcal{M}$ for every positive vector $\mathbf{t}\in \mathbb{Z}^{n}$ (equivalently, $\mathcal{M}$  is closed under blow-up).
 \end{defi}

 \begin{defi}[Flat property]
 A hereditary property $\mathcal{P}$  is  \emph{flat} if $\lambda^{(1)}(\mathcal{P})=\hat{\pi}(\mathcal{P})>0$ (where $\hat{\pi}(\mathcal{P})>0$ follows from the Remark after Lemma \ref{exi}).
\end{defi}

 \begin{lem}[\hspace{1sp}\cite{N2014A,N2014B}]\label{flat}
If $\mathcal{P}$ is a hereditary, multiplicative property, then it is flat; that is to say,  $\lambda^{(\alpha)}(\mathcal{P})=\hat{\pi}(\mathcal{P})$ for every $\alpha\geq 1$.
\end{lem}

 The flat property yields significant  advantages in extremal hypergraph  problems, as shown by the following results.

\begin{lem}[\hspace{1sp}\cite{N2014A,N2014B}]\label{ttt6}
 If $\mathcal{P}$ is a flat property of $r$-graphs, and $G\in \mathcal{P}$, then for every $\alpha\geq1$,
$$\lambda^{(\alpha)}(G)\leq\hat{\pi}(\mathcal{P})^{1/\alpha}(r!e(G))^{1-1/\alpha}.$$
\end{lem}

 %there exists  concise and tight upper
 %bounds on $e(G)$ and $\lambda^{(\alpha)}(G)$ for every $G\in \mathcal{H}$.

 \begin{thm}[\hspace{1sp}\cite{N2014A,N2014B}]\label{t8}
If $\mathcal{P}$ is a flat property of $r$-graphs, and $G\in \mathcal{P}_{n}$, then
\begin{equation*}
e(G)\leq\hat{\pi}(\mathcal{P})n^{r}/r!,
\end{equation*}
and for every $\alpha\geq1$,
\begin{equation*}
\lambda^{(\alpha)}(G)\leq\hat{\pi}(\mathcal{P})n^{r-r/\alpha}.
\end{equation*}
\end{thm}

\noindent {\bf Remark.} It is worth mentioning that Theorem \ref{t8} first appeared in \cite[Theorem $16$]{N2014A}. However, its proof  is based on a stronger condition that $\mathcal{P}$ is both hereditary and multiplicative property  (see \cite[p. $27$]{N2014A}).  We therefore provide a complete proof in Appendix.

\section{Asymptotic regularity of spectral extremal hypergraphs}

In this section, we establish a  notable property of spectral extremal hypergraphs  in pattern-colorable $r$-graphs. Specifically, given an $r$-pattern $P$,  if $G$ is a $P$-colorable $r$-graph of order $n$ attaining the maximum $\alpha$-spectral radius for $\alpha>1$,  then $G$ becomes almost regular as $n\to \infty$.
To prove this,   we  first consider a broader family of $r$-graphs with hereditary property.

\subsection{Spectral extremal hypergraphs  of hereditary property}

 Nikiforov derived the following result  through an ingenious and concise method (see Claim A in the proof of \cite[Theorem $12$]{N2014A}).

\begin{lem}[\hspace{1sp}\cite{N2014A}]\label{t3}
Let $\mathcal{P}$ be a hereditary property of $r$-graphs with $\hat{\pi}(\mathcal{P})>0$. If $\alpha>1$ and $\lambda_{n}=
\lambda^{(\alpha)}(\mathcal{P},n)$, then there exist infinitely many $n$ such that
$$\frac{\lambda_{n-1}(n-1)^{r/\alpha-1}}{(n-2)_{r-1}}-\frac{\lambda_{n}n^{r/\alpha-1}}{(n-1)_{r-1}}<\frac{1}{n\log n}\cdot\frac{\lambda_{n}n^{r/\alpha-1}}{(n-1)_{r-1}}.$$
\end{lem}

Let $G\in \mathcal{P}_{n}$ be an $r$-graph  satisfying $\lambda^{(\alpha)}(G)=\lambda^{(\alpha)}(\mathcal{P},n)$.  Applying  Lemma \ref{t3}, Nikiforov \cite{N2014A} showed that for $1<\alpha\leq r$ and $n$ sufficiently large, there exist infinitely many $n$ such that
$$(x_{\textup{min}}(G))^{\alpha}\geq\frac{1}{n}\Big(1-\frac{1}{(\alpha-1)\log n}\Big).$$
However, a similar argument extends this result to  any $\alpha>1$. We require the following elementary fact before presenting Lemma \ref{t4}.\\

\noindent {\bf Fact 1.} If $\alpha>1$ and $r\geq2$, then the function
$$f(x)=\frac{1-rx}{(1-x)^{r/\alpha}}$$
is decreasing for $0\leq x<1$.

\begin{proof}
Computing the derivative of  $f(x)$, we obtain
\begin{displaymath}
\begin{split}
f'(x)&=\frac{-r(1-x)^{r/\alpha}+(r/\alpha)(1-rx)(1-x)^{r/\alpha-1}}{(1-x)^{2r/\alpha}}\\
&=\frac{r(1-x)^{r/\alpha-1}}{\alpha(1-x)^{2r/\alpha}}\big(-(\alpha-1)-(r-\alpha)x\big).
\end{split}
\end{displaymath}
If $1<\alpha\leq r$, then $f'(x)<0$ is immediate. For  $\alpha>r$, observe that
\begin{displaymath}
\begin{split}
-(\alpha-1)-(r-\alpha)x<-(\alpha-1)-(r-\alpha)=1-r<0
\end{split}
\end{displaymath}
since $x<1$ and $r\geq2$. Thus,  $f'(x)<0$ holds in all cases.
\end{proof}

For an $r$-graph $H$ and a vertex $v\in V(H)$, let  $E_{H}(v)$  denote the set of edges in $H$ containing $v$, and let $d_{H}(v)=|E_{H}(v)|$ be the degree of $v$.  The minimum degree of $H$ is denoted by
$\delta(H)$.
For a vertex subset $U\subseteq V(H)$, we write $x_{U}=\Pi_{v\in U}x_{v}$.

\begin{lem}\label{t4}
Let $\mathcal{P}$ be a hereditary  property of $r$-graphs with $\hat{\pi}(\mathcal{P})>0$. Suppose that $G_{n}\in \mathcal{P}_{n}$ is an $r$-graph satisfying $\lambda^{(\alpha)}(G_{n})=\lambda^{(\alpha)}(\mathcal{P},n)$ for $\alpha>1$.
Then there exist infinitely many $n$ such that
$$(x_{\textup{min}}(G_n))^{\alpha}\geq\frac{1}{n}\Big(1-\frac{\alpha}{(\alpha-1)r\log n}\Big).$$
\end{lem}
\begin{proof}
Assume, for contradiction,  that there exists $n_{0}$ such that for all $n>n_{0}$,
$$(x_{\textup{min}}(G_n))^{\alpha}<\frac{1}{n}\Big(1-\frac{\alpha}{(\alpha-1)r\log n}\Big).$$

Set  $\lambda_{n}:=\lambda^{(\alpha)}(\mathcal{P},n)$. By Lemma \ref{t3}, we can select $n>n_{0}$ sufficiently large so that
$$\frac{\lambda_{n-1}(n-1)^{r/\alpha-1}}{(n-2)_{r-1}}-\frac{\lambda_{n}n^{r/\alpha-1}}{(n-1)_{r-1}}<\frac{1}{n\log n}\cdot\frac{\lambda_{n}n^{r/\alpha-1}}{(n-1)_{r-1}},$$
which implies
\begin{equation}\label{e7}
\frac{\lambda_{n-1}}{\lambda_{n}}<\frac{n^{r/\alpha-1}(n-r)}{(n-1)^{r/\alpha}}\Big(1+\frac{1}{n\log n}\Big).
\end{equation}
By the definition of $x_{\textup{min}}(G_n)$,  there exists
   a principal eigenvector $\mathbf{x}$ of $G_n$ such that
   \begin{equation}\label{e6}
(\mathbf{x}_{\textup{min}})^{\alpha}<\frac{1}{n}\Big(1-\frac{\alpha}{(\alpha-1)r\log n}\Big).
\end{equation}
Let $k\in V(G_{n})$ be a vertex with $x_{k}=\mathbf{x}_{\textup{min}}$, and let $\mathbf{x}'$ be the $(n-1)$-vector obtained from $\mathbf{x}$ by removing the component $x_{k}$. For the subgraph
$G_{n}-k$, we have
$$P_{G_{n}-k}(\mathbf{x}')=\lambda^{(\alpha)}(G_{n})-r!x_{k}\sum_{e\in E_{G_{n}}(k)}x_{e \backslash\{k\}}=
\lambda_{n}-r\lambda_{n}x_{k}^{\alpha}.$$
Since $\mathcal{P}$ is hereditary, $G_{n}-k\in \mathcal{P}_{n-1}$. Therefore,
$$\lambda_{n}(1-rx_{k}^{\alpha})=P_{G_{n}-k}(\mathbf{x}')\leq \lambda^{(\alpha)}(G_{n}-k)(\|\mathbf{x}' \|_{\alpha})^{r}
\leq \lambda_{n-1}(1-x_{k}^{\alpha})^{r/\alpha},$$
or equivalently,
\begin{equation}\label{e8}
\frac{\lambda_{n-1}}{\lambda_{n}}\geq \frac{1-rx_{k}^{\alpha}}{(1-x_{k}^{\alpha})^{r/\alpha}}.
\end{equation}
Combining (\ref{e7}) and (\ref{e8}) yields
$$\frac{1-rx_{k}^{\alpha}}{(1-x_{k}^{\alpha})^{r/\alpha}}\leq \frac{\lambda_{n-1}}{\lambda_{n}}
\leq \frac{n^{r/\alpha-1}(n-r)}{(n-1)^{r/\alpha}}\Big(1+\frac{1}{n\log n}\Big).$$
Applying Fact $1$ and incorporating (\ref{e6}), we derive
$$\frac{1-\frac{r}{n}(1-\frac{\alpha}{(\alpha-1)r\log n})}{\Big({1-\frac{1}{n}(1-\frac{\alpha}{(\alpha-1)r\log n})}\Big)^{r/\alpha}}
\leq \frac{1-rx_{k}^{\alpha}}{(1-x_{k}^{\alpha})^{r/\alpha}}\leq\frac{n^{r/\alpha-1}(n-r)}{(n-1)^{r/\alpha}}\Big(1+\frac{1}{n\log n}\Big),$$
and then
$$\frac{\Big(n-r+\frac{\alpha}{(\alpha-1)\log n}\Big)n^{r/\alpha-1}}
{\Big(n-1+\frac{\alpha}{(\alpha-1)r\log n}\Big)^{r/\alpha}}
\leq \frac{1-rx_{k}^{\alpha}}{(1-x_{k}^{\alpha})^{r/\alpha}}\leq\frac{n^{r/\alpha-1}(n-r)}{(n-1)^{r/\alpha}}\Big(1+\frac{1}{n\log n}\Big).$$
Further simplification leads to
$$1+\frac{\alpha}{(\alpha-1)(n-r)\log n}\leq\Big(1+\frac{\alpha}{(\alpha-1)r(n-1)\log n}\Big)^{r/\alpha}\Big(1+\frac{1}{n\log n}\Big).$$
For sufficiently large $n$, we have
\begin{displaymath}
\begin{split}
\Big(1+\frac{\alpha}{(\alpha-1)r(n-1)\log n}\Big)^{r/\alpha}
&=1+\frac{1}{(\alpha-1)(n-1)\log n}+O\Big(\frac{1}{(n\log n)^{2}}\Big)\\
&\leq1+\frac{1}{(\alpha-1)(n-1)\log n}+\frac{1}{(\alpha-1)(n-1)(n-2)\log n}\\
&=1+\frac{1}{(\alpha-1)(n-2)\log n}.
\end{split}
\end{displaymath}
Substituting this bound, we obtain
$$1+\frac{\alpha}{(\alpha-1)(n-r)\log n}\leq \Big(1+\frac{1}{(\alpha-1)(n-2)\log n}\Big)\Big(1+\frac{1}{n\log n}\Big).$$
By some cancellations and rearranging, we get
$$\frac{\alpha}{n-r}\leq \frac{1}{n-2}+\frac{\alpha-1}{n}+\frac{1}{n(n-2)\log n}.$$
Noting that $\frac{\alpha}{n-r}\geq\frac{\alpha}{n-2}$, this inequality  reduces to
$$2(\alpha-1)\leq\frac{1}{\log n},$$
which is impossible for sufficiently large $n$.
\end{proof}

Furthermore, we show that if the minimum component of a principal eigenvector of the spectral extremal hypergraph
approaches  $n^{-1/\alpha}$, then its minimum degree  approximates  its average degree,  implying that the spectral extremal hypergraph is asymptotically regular.

\begin{lem}\label{t2}
Let  $\mathcal{P}$ be a hereditary  property of $r$-graphs with $\hat{\pi}(\mathcal{P})>0$, and  $G_{n}\in \mathcal{P}_{n}$ satisfy $\lambda^{(\alpha)}(G_{n})=\lambda^{(\alpha)}(\mathcal{P},n)$ for $\alpha>1$. Let $0<\varepsilon<1$, $0\leq\varepsilon'<\varepsilon \hat{\pi}(\mathcal{P})/(r-1)$, and $\mathbf{x}\in \mathbb{R}^{n}$ be a principal eigenvector  of $G_{n}$. If
$n$ is sufficiently large and
\[(\mathbf{x}_{\textup{min}})^{\alpha}\geq \frac{1}{n}(1-\varepsilon'),\]
then
\[\delta(G_{n})\geq (1-\varepsilon)\hat{\pi}(\mathcal{P})\binom{n}{r-1}.\]
\end{lem}
\begin{proof}
Set  $V=V(G_{n})$, $\lambda=\lambda^{(\alpha)}(G_{n})$, $\delta=\delta(G_{n})$, $\mu_n=\mathbf{x}_{\textup{min}}$, and let $k\in V$ be a vertex attaining minimum degree $\delta$. Considering the eigenequation for $\lambda^{(\alpha)}(G_{n})$  at  vertex  $k$:
$$\lambda \mu_n^{\alpha-1}\leq \lambda x^{\alpha-1}_{k}=(r-1)!\sum_{e\in E_{G_{n}}(k)}x_{e\backslash\{k\}}.$$
 By H\"older's inequality,  we have
\begin{equation}\label{e1}
\bigg(\frac{\lambda \mu_n^{\alpha-1}}{(r-1)!}\bigg)^{\alpha}\leq \delta^{\alpha-1}\sum_{e\in E_{G_{n}}(k)}x_{e\backslash\{k\}}^{\alpha}.
\end{equation}
First, observe that
\begin{align}\label{e2}
\begin{split}
\sum_{e\in E_{G_{n}}(k)}x_{e \backslash\{k\}}^{\alpha}
&=\sum_{S\in \binom{V}{r-1}}x_{S}^{\alpha}-\sum_{T\in \binom{V}{r-1}~and~T\cup \{k\} \notin E_{G_{n}}(k) }x_{T}^{\alpha}\\
&\leq \sum_{S\in \binom{V}{r-1}}x_{S}^{\alpha}-\sum_{T\in \binom{V}{r-1}~and~T\cup \{k\} \notin E_{G_{n}}(k) }\mu_n^{\alpha(r-1)}\\
&=\sum_{S\in \binom{V}{r-1}}x_{S}^{\alpha}-\big(\tbinom{n}{r-1}-\delta\big)\mu_n^{\alpha(r-1)}.
\end{split}
\end{align}
By Maclaurin's inequality, we have
\begin{equation}\label{e3}
\sum_{S\in \binom{V}{r-1}}x_{S}^{\alpha}\leq \binom{n}{r-1}\Bigg(n^{-1}\sum_{i\in V}x_{i}^{\alpha}\Bigg)^{r-1}=
\frac{\binom{n}{r-1}}{n^{r-1}}.
\end{equation}
Suppose for  contradiction that $\delta(G_{n})<(1-\varepsilon)\hat{\pi}(\mathcal{P})\binom{n}{r-1}$. Combining (\ref{e2}) and (\ref{e3}), we obtain
\begin{align}\label{e4}
\begin{split}
\sum_{e\in E_{G_{n}}(k)}x_{e \backslash\{k\}}^{\alpha}
&\leq\frac{\binom{n}{r-1}}{n^{r-1}}-\big(1-(1-\varepsilon)\hat{\pi}(\mathcal{P})\big)\tbinom{n}{r-1}\mu_n^{\alpha(r-1)}\\
&\leq\frac{\binom{n}{r-1}}{n^{r-1}}-\big(1-(1-\varepsilon)\hat{\pi}(\mathcal{P})\big)\tbinom{n}{r-1}\frac{(1-\varepsilon')^{r-1}}{n^{r-1}}\\
&=\frac{\binom{n}{r-1}}{n^{r-1}}\Big(1-\big(1-(1-\varepsilon)\hat{\pi}(\mathcal{P})\big)\big(1-\varepsilon'\big)^{r-1}\Big).
\end{split}
\end{align}
 By  Bernoulli's inequality and the definition of $\varepsilon'$,   we have
 \begin{displaymath}
\begin{split}
(1-\varepsilon')^{r-1}\geq 1-(r-1)\varepsilon' \geq 1-\varepsilon \hat{\pi}(\mathcal{P}).
%\frac{(1-\varepsilon')^{r-1}\binom{n}{r-1}}{n^{r-1}}
%&\geq\frac{(1-(r-1)\varepsilon')\binom{n}{r-1}}{n^{r-1}}\\
%&\geq\frac{1}{(r-1)!}\big(1-\frac{\varepsilon \hat{\pi}(\mathcal{P})}{2}\big)\frac{(n)_{r-1}}{n^{r-1}}\\
%&\geq \frac{1}{(r-1)!}(1-\varepsilon \hat{\pi}(\mathcal{P})),
\end{split}
\end{displaymath}
Therefore,
\begin{align}\label{e5}
\begin{split}
\sum_{e\in E_{G_{n}}(k)}x_{e \backslash\{k\}}^{\alpha}
&\leq \frac{\binom{n}{r-1}}{n^{r-1}}\Big(1-\big(1-(1-\varepsilon)\hat{\pi}(\mathcal{P})\big)\big(1-\varepsilon \hat{\pi}(\mathcal{P})\big)\Big)\\
&= \frac{\binom{n}{r-1}}{n^{r-1}}\Big(\hat{\pi}(\mathcal{P})-\varepsilon(1-\varepsilon)\hat{\pi}(\mathcal{P})^{2}\Big)\\
&\leq\frac{\hat{\pi}(\mathcal{P})\binom{n}{r-1}}{n^{r-1}}.
\end{split}
\end{align}

On the other hand,  by Lemmas \ref{exi} and   \ref{t1},  for sufficiently large  $n$,
$$\lambda\geq (1-r^{2}/n)\lambda^{(\alpha)}(\mathcal{P})n^{r(\alpha-1)/\alpha}=
(1-r^{2}/n)\hat{\pi}(\mathcal{P})n^{r(\alpha-1)/\alpha}.$$
Combining this with (\ref{e1}) and (\ref{e5}), we obtain
$$\frac{(1-r^{2}/n)^{\alpha}\hat{\pi}(\mathcal{P})^{\alpha}n^{r(\alpha-1)}}{((r-1)!)^{\alpha}}\mu_n^{\alpha(\alpha-1)}\leq
\bigg(\frac{\lambda \mu_n^{\alpha-1}}{(r-1)!}\bigg)^{\alpha}\leq \frac{\hat{\pi}(\mathcal{P})\delta^{\alpha-1}\binom{n}{r-1}}{n^{r-1}}.$$
Since $\mu_n^{\alpha}\geq(1-\varepsilon')/n$ and  $\delta<(1-\varepsilon)\hat{\pi}(\mathcal{P})\binom{n}{r-1}$,   we arrive at
$$\frac{(1-\varepsilon')^{\alpha-1}(1-r^{2}/n)^{\alpha}\hat{\pi}(\mathcal{P})^{\alpha}n^{(r-1)(\alpha-1)}}{((r-1)!)^{\alpha}}
\leq\frac{(1-\varepsilon)^{\alpha-1}\hat{\pi}(\mathcal{P})^{\alpha}((n)_{r-1})^{\alpha}}{((r-1)!)^{\alpha}n^{r-1}}.$$
After some eliminations and scalings, we have
$$(1-\varepsilon')^{\alpha-1}(1-r^{2}/n)^{\alpha}\leq (1-\varepsilon)^{\alpha-1},$$
that is
\begin{equation}\label{qj}
1-\frac{r^{2}}{n}\leq \big(\frac{1-\varepsilon}{1-\varepsilon'}\big)^{\frac{\alpha-1}{\alpha}}.
\end{equation}
Thus,
$$0<1-\big(\frac{1-\varepsilon}{1-\varepsilon'}\big)^{\frac{\alpha-1}{\alpha}}\leq\frac{r^{2}}{n},$$
which is a contradiction. This completes the proof of Lemma \ref{t2}.
\end{proof}

\subsection{Spectral extremal hypergraphs in pattern-colorable hypergraphs}

Building upon  properties of spectral extremal hypergraphs for the hereditary property, we establish  lower bounds on the minimum component of the principal eigenvector and the minimum degree of spectral extremal hypergraphs in pattern-colorable hypergraphs. The following concepts and results are fundamental.

\begin{defi}[Clonal family\cite{CDS2024}]
Let $G$ be an $r$-graph and let $u, v\in V(G)$ be two vertices. Denote by $G_{u\to v}$ the $r$-graph obtained as follows: first delete all edges incident to $u$, then for every edge $e\in E(G)$ with $v\in e$ and $u\notin e$, add the edge $e+u-v$. A family $\mathcal{F}$ of $r$-graphs is called \emph{clonal} if for every $G\in\mathcal{F}$ and all $u, v\in V(G)$ we have $G_{u\to v}\in\mathcal{F}$.
 \end{defi}

\begin{defi}[Principal ratio]
For $\alpha>1$, and a principal eigenvector $\mathbf{x}$ of an $r$-graph $G$,   the \emph{principal ratio} of $(G,\mathbf{x})$ is
$$\gamma(G,\mathbf{x}):=\frac{\mathbf{x}_{\textup{max}}}{\mathbf{x}_{\textup{min}}},$$
where $\mathbf{x}_{\textup{max}}$ is the maximum component of the vector $\mathbf{x}$.
If $\mathbf{x}$ has a zero component, then we set $\gamma(G,\mathbf{x})=\infty$.
\end{defi}

In \cite{CDS2024}, Cooper, Desai, and Sahay studied the principal eigenvector of extremal graphs for the clonal family of  $r$-graphs, and showed that the principal ratio is close to $1$.

\begin{lem}[\hspace{1sp}\cite{CDS2024}]\label{clonal}
Let $\alpha>1$ and $\mathcal{F}$ be a clonal family of $r$-graphs with $\hat{\pi}(\mathcal{F})>0$.  Suppose that $G\in \mathcal{F}_{n}$  such that $\lambda^{(\alpha)}(G)=\lambda^{(\alpha)}(\mathcal{F},n)$.
If $\mathbf{x}$ is a Perron-Frobenius eigenvector of $G$, then
$$\gamma(G,\mathbf{x})=1+O(n^{-1}).$$
\end{lem}

\noindent {\bf Remark.} The definitions of parameters $\hat{\pi}(\mathcal{F})$, $\mathcal{F}_{n}$ and $\lambda^{(\alpha)}(\mathcal{F},n)$ for the
clonal (or general) family of $r$-graphs are the same as those for the hereditary  property of $r$-graphs.

\begin{lem}[\hspace{1sp}\cite{KNY2014}]\label{aut}
Let $G$ be an  $r$-graph of order $n$ with at least one edge, and let $u$ and $v$ be vertices of $G$ such that  transposing  $u$ and $v$ is an automorphism. If $\alpha>1$, and $\mathbf{x}$ is an eigenvector corresponding to $\lambda^{(\alpha)}(G)$, then $x_{u}=x_{v}$.
\end{lem}

\begin{lem}[\hspace{1sp}\cite{KNY2014}]\label{nonne}
Let $\alpha\geq1$, and let $G$ be an  $r$-graph such that every nonnegative eigenvector corresponding to $\lambda^{(\alpha)}(G)$ is positive. If $H$ is a subgraph of $G$, then $\lambda^{(\alpha)}(H)<\lambda^{(\alpha)}(G)$, unless $H=G$.
\end{lem}

%The family of all $P$-colorable $r$-graphs is a monotone property.
Recall that $\mathrm{Col}(P)$ denotes  the family of all $P$-colorable $r$-graphs.
 It is not hard to show that $\mathrm{Col}(P)$ is hereditary and multiplicative (see Fact A in Appendix).  Lemma \ref{flat} implies that $\mathrm{Col}(P)$  is flat, and hence  $\hat{\pi}(\mathrm{Col}(P))>0$. Some spectral extremal problems for hypergraph classes can be reduced to the corresponding problems in $\mathrm{Col}(P)$ for suitable  $P$ (cf. Theorem \ref{chro}).
To characterize the properties of the spectral extremal hypergraphs in $\mathrm{Col}(P)$, we require the following  lemma.
%Recall that $\mathrm{Col}(P)_{n}$ stands for the set of all $r$-graphs in $\mathrm{Col}(P)$ of order $n$.

\begin{lem}\label{t5}
Let $\alpha>1$, $r\geq2$ and $1\leq i\leq r$.  If $m$ is sufficiently large, then
$$\binom{m+1}{i}\Big(\frac{m}{m+1}\Big)^{i/\alpha}-\binom{m}{i}\geq\binom{m-1}{i-1}
\Big(1-\frac{1}{\alpha}-\frac{1}{\alpha(m-r+1)}\Big).$$
\end{lem}

\begin{proof}
For  enough large $m$, we have
\begin{displaymath}
\begin{split}
\Big(\frac{m}{m+1}\Big)^{i/\alpha}
&=\Big(1-\frac{1}{m+1}\Big)^{i/\alpha}\\
&=1-\frac{i}{\alpha(m+1)}+\frac{\binom{i/\alpha}{2}}{(m+1)^{2}}
+o\big(m^{-2}\big)\\
&=1-\frac{i}{\alpha m}+\frac{i}{\alpha m(m+1)}+\frac{\binom{i/\alpha}{2}}{(m+1)^{2}}
+o\big(m^{-2}\big).
\end{split}
\end{displaymath}
Since  $\frac{i}{\alpha}+\binom{i/\alpha}{2}>0$, it follows that
$$\frac{i}{\alpha m(m+1)}+\frac{\binom{i/\alpha}{2}}{(m+1)^{2}}
+o\big(m^{-2}\big)>0.$$
Therefore,
\begin{equation}\label{e9}
\Big(\frac{m}{m+1}\Big)^{i/\alpha}>1-\frac{i}{\alpha m}.
\end{equation}

On the  other  hand, we have
\begin{align}\label{e10}
\begin{split}
\binom{m+1}{i}\Big(1-\frac{i}{\alpha m}\Big)-\binom{m}{i}
&=\binom{m}{i-1}-\frac{m+1}{\alpha(m-i+1)}\binom{m-1}{i-1}\\
&=\frac{m}{m-i+1}\binom{m-1}{i-1}-\frac{m+1}{\alpha(m-i+1)}\binom{m-1}{i-1}\\
&=\binom{m-1}{i-1}
\Big(1-\frac{1}{\alpha}+\frac{\alpha(i-1)-i}{\alpha(m-i+1)}\Big)\\
&\geq \binom{m-1}{i-1}
\Big(1-\frac{1}{\alpha}-\frac{1}{\alpha(m-i+1)}\Big)\\
&\geq \binom{m-1}{i-1}
\Big(1-\frac{1}{\alpha}-\frac{1}{\alpha(m-r+1)}\Big).
\end{split}
\end{align}
The conclusion follows immediately from (\ref{e9}) and (\ref{e10}).
\end{proof}

\begin{thm}\label{Co}
Let  $P=([l],E)$ be an  $r$-pattern, and  $G_{n}\in \mathrm{Col}(P)_{n}$ satisfy $\lambda^{(\alpha)}(G_{n})=\lambda^{(\alpha)}(\mathrm{Col}(P),n)$. If $\alpha>1$, then there exist $n_{1}\in \mathbb{N}$ and a constant $M>0$  such that for any $n\geq n_{1}$, the following hold:

$(1)$ $\lambda^{(\alpha)}(G_{n+1}) \ge  \big(1 + r \big(1 - \frac{1}{\alpha} - \frac{l}{\alpha(n - l r + l)}\big) ( x_{\textup{min}}(G_n) )^{\alpha} \big)\lambda^{(\alpha)}(G_n).$

$(2)$  $(x_{\textup{min}}(G_n))^{\alpha} \ge \frac{1}{n}\big( 1 - \frac{\hat{\pi}(\mathrm{Col}(P)) M}{r n} \big).$

$(3)$ $\delta(G_n) \ge \hat{\pi}(\mathrm{Col}(P)) \left(1 - \frac{M}{n}\right) \binom{n}{r-1}.$
\end{thm}
\begin{proof}
%Let $N$ be sufficiently large. By Lemma \ref{t4}, there is $n_{0}>N$ such that
%$$x^{\alpha}>\frac{1}{n_{0}}\Big(1-\frac{\alpha}{(\alpha-1)r\log n_{0}}\Big)\geq\frac{1}{n_{0}}(1-\varepsilon'),$$
%and so
%$$\delta(G_{n_{0}})\geq (1-\varepsilon)\hat{\pi}(\mathrm{Col}(P))\binom{n_{0}}{r-1}$$
%by Theorem \ref{t2}.

%Now we prove the theorem by induction, suppose that the result holds for $n$ ($n\geq n_{0}$), then
%$$x^{\alpha}\geq \frac{1}{n}(1-\varepsilon').$$
$(1)$ For any given  $n\in \mathbb{N}$, select $G_{n}\in \mathrm{Col}(P)_{n}$  satisfying $\lambda^{(\alpha)}(G_{n})=\lambda^{(\alpha)}(\mathrm{Col}(P),n)$. Clearly, there exists an edge-maximal $r$-graph $H_n\in \mathrm{Col}(P)_{n}$ such that $G_{n}\subseteq H_{n}$.  Observe that $\lambda^{(\alpha)}(G_{n})\leq\lambda^{(\alpha)}(H_{n})$, and thus $\lambda^{(\alpha)}(G_{n})=\lambda^{(\alpha)}(H_{n})$.
 Let $\phi$ be a homomorphism from $H_{n}$ to $P$.  For any  $i\in [l]$, and vertices $u$, $v\in\phi^{-1}(i)$,  if an edge $e\in E(H_{n})$ contains $u$ but not $v$, then $e-u+v\in E(H_{n})$ by  edge-maximality. Consequently,  transposing  $u$ and $v$ induces an automorphism of $H_{n}$.

Clearly, there exists some $j\in[l]$ such that $m:=|\phi^{-1}(j)|\geq n/l$. For sufficiently large $n$,  $m$ is correspondingly large. Fix a vertex $w\in \phi^{-1}(j)$ and construct an $r$-graph $H_{n}\circ k$  of order $n+1$ obtained from  $H_{n}$ by adding a vertex $k$, with a homomorphism $\phi'$ extending $\phi$  such that $\phi'(k)=\phi(w)$. The edge set is given by: $$E(H_{n}\circ k)=E(H_{n})\cup \{e-v+k: e\in E(H_{n}), v\in  e\cap\phi^{-1}(j)\}.$$
This construction ensures that $H_{n}\circ k$ is $P$-colorable and edge-maximal, and transposing  $k$ and $w$ yields  an automorphism of $H_{n}\circ k$.

Let $\mathbf{x}\in \mathbb{R}^{n}$ be a principal eigenvector of $H_{n}$ and $\mu_n:=\mathbf{x}_{\textup{min}}$.
Define a vector $\mathbf{y}\in \mathbb{R}^{n+1}$   as follows:
$$y_{v}=\left\{
\begin{array}{cl}
\Big(\frac{mx_{w}^{\alpha}}{m+1}\Big)^{1/\alpha}, & \mbox{~if~} v \in \phi^{-1}(j)\cup \{k\};\\
x_{v}, & \mbox{~else}.
\end{array}\right.
$$
By Lemma \ref{aut}, $x_{i}=x_{w}$  for all $i\in \phi^{-1}(j)$, which implies $\sum_{v\in V(H_{n}\circ k)}y_{v}^{\alpha}=1$. Set $\lambda_{n}:=\lambda^{(\alpha)}(\mathrm{Col}(P), n)$; then $\lambda_{ n+1}\geq\lambda^{(\alpha)}(H_{n}\circ k)$.

%\noindent\textbf{Claim.}
%If $n$ is sufficiently large, then %$\lambda_{n+1}\geq\big(1+r(1-\frac{1}{\alpha}-\frac{l}%{\alpha(n-lr+l)})(x_{min}(G_n))^{\alpha}\big)\lambda_{n}$.
\begin{myCli}\label{1}
There exists $n_{0}\in \mathbb{N} $ such that $\lambda_{n+1}\geq\big(1+r(1-\frac{1}{\alpha}-\frac{l}{\alpha(n-lr+l)})\mu_{n}^{\alpha}\big)\lambda_{n}$ for all $n> n_{0}$.
\end{myCli}

\begin{proof}
 Partition the edges incident to vertex $w$ into sets $E_{H_{n}}^{(i)}(w)$ defined by $$E_{H_{n}}^{(i)}(w)=\{e\in E_{H_{n}}(w): |e\cap \phi^{-1}(j)|=i\}$$ (which may be empty). For $1\leq i\leq r$, define $M_{i}=\{e\backslash\phi^{-1}(j):
e\in E_{H_{n}}^{(i)}(w)\}$. It follows that  $E_{H_{n}}(w)=\cup_{i=1}^{r}E_{H_{n}}^{(i)}(w)$, and for each $i$ ($1\leq i \leq r$),
$$\sum_{e\in E_{H_{n}}^{(i)}(w)}x_{e\backslash\{w\}}=\binom{m-1}{i-1}x_{w}^{i-1}\sum_{F\in M_{i}}x_{F}.$$
By Lemmas \ref{aut} and \ref{t5}, there exists $n_{0}\in \mathbb{N}$ such that  for all $n>n_0$,
\begin{align*}
\lambda_{n+1}-\lambda_{n}
&\geq P_{H_n\circ k}(\mathbf{y})-P_{H_n}(\mathbf{x})\\
&=r!\sum_{i=1}^{r}\Bigg(\binom{m+1}{i}\Big(\frac{m x_{w}^{\alpha}}{m+1}\Big)^{\frac{i}{\alpha}}-\binom{m}{i}x_{w}^{i}\Bigg)
\sum_{F\in M_{i}}x_{F}\\
&\geq r!\sum_{i=1}^{r}\binom{m-1}{i-1}\Big(1-\frac{1}{\alpha}-\frac{1}{\alpha(m-r+1)}\Big)x_{w}^{i}\sum_{F\in M_{i}}x_{F}\\
%\displaybreak[1]
&=r!\sum_{i=1}^{r}x_{w}\Big(1-\frac{1}{\alpha}-\frac{1}{\alpha(m-r+1)}\Big)\binom{m-1}{i-1}x_{w}^{i-1}\sum_{F\in M_{i}}x_{F}\\
&=r!x_{w}\Big(1-\frac{1}{\alpha}-\frac{1}{\alpha(m-r+1)}\Big)\sum_{i=1}^{r}\sum_{e\in E_{H_{n}}^{(i)}(w)}x_{e\backslash\{w\}}\\
&=r!x_{w}\Big(1-\frac{1}{\alpha}-\frac{1}{\alpha(m-r+1)}\Big)\sum_{e\in E_{H_{n}}(w)}x_{e\backslash\{w\}}\\
&=r!x_{w}\Big(1-\frac{1}{\alpha}-\frac{1}{\alpha(m-r+1)}\Big)\cdot \frac{\lambda_{n}x_{w}^{\alpha-1}}{(r-1)!}\\
&=r\Big(1-\frac{1}{\alpha}-\frac{1}{\alpha(m-r+1)}\Big)\lambda_{n}x_{w}^{\alpha}\\
&\geq r\Big(1-\frac{1}{\alpha}-\frac{l}{\alpha(n-lr+l)}\Big)\lambda_{n}\mu_n^{\alpha}.
%&\geq r\Big(1-\frac{1}{\alpha}-\frac{l}{\alpha(n-%lr+l)}\Big)\lambda_{n}\mu_n^{\alpha}.
 \end{align*}
 This  yields that  $\lambda_{n+1}\geq\big(1+r(1-\frac{1}{\alpha}-\frac{l}{\alpha(n-lr+l)})\mu_{n}^{\alpha}\big)\lambda_{n}$
as desired.
\end{proof}
\begin{myCli}\label{2}There exists $n_{1}$ such that
for all $n>n_{1}$, every principal eigenvector of $H_{n}$  is strictly positive.\end{myCli}
\begin{proof}
Recall that $\lambda^{(\alpha)}(H_{n})=\lambda^{(\alpha)}(\mathrm{Col}(P),n)$ and $\hat{\pi}(\mathrm{Col}(P))>0$. By  Lemma \ref{t4},   there exists sufficiently large   $n_{1}>n_0$ such that
$$\mu_{n_{1}}^{\alpha}\geq(x_{\textup{min}}(H_{n_1}))^{\alpha}\geq\frac{1}{n_{1}}\Big(1-\frac{\alpha}{(\alpha-1)r\log n_{1}}\Big)>0.$$
By Claim $1$, we have
$$\lambda_{n_{1}+1}-\lambda_{n_{1}}\geq r\Big(1-\frac{1}{\alpha}-\frac{l}{\alpha(n_{1}-lr+l)}\Big)\lambda_{n_{1}}\mu_{n_{1}}^{\alpha}>0.$$
We now show that  every principal eigenvector of  $H_{n_{1}+1}$  is strictly positive. Suppose, for contradiction, that there exist a principal eigenvector $\mathbf{z}$ of $H_{n_{1}+1}$ and a vertex $v\in V(H_{n_{1}+1})$ such that $z_{v}=0$.  Let $\mathbf{z}'$ be the vector obtained from $\mathbf{z}$ by omitting $z_{v}$. Then $\|\mathbf{z}'\|_{\alpha}=1$, and
$$\lambda_{n_{1}}\geq P_{H_{n_{1}+1}-v}(\mathbf{z}')=P_{H_{n_{1}+1}}(\mathbf{z})=\lambda_{n_{1}+1},$$
which contradicts $\lambda_{n_{1}+1}>\lambda_{n_{1}}$.
Proceeding by induction, assume that the
claim holds for some $n\geq n_{1}+1$, so that
every principal eigenvector of  $H_{n}$  must be strictly positive (hence $\mu_{n}^{\alpha}>0$). Claim $1$ implies $\lambda_{n+1}-\lambda_{n}>0$. The same contradiction argument as above shows that every principal eigenvector of  $H_{n+1}$  is strictly positive. By induction, the claim holds for all $n>n_{1}$.
\end{proof}
Claim $2$ and Lemma \ref{nonne}  imply  $G_{n}=H_n$    for  $n>n_{1}$. By Claim $1$, we have
\begin{align*}
\lambda_{n+1}&\geq\Big(1+r\Big(1-\frac{1}{\alpha}-\frac{l}{\alpha(n-lr+l)}\Big)\mu_{n}^{\alpha}\Big)\lambda_{n}\\
&\geq \Big(1+r\Big(1-\frac{1}{\alpha}-\frac{l}{\alpha(n-lr+l)}\Big)( x_{\textup{min}}(H_n) )^{\alpha} \Big)\lambda_{n}\\
&=\Big(1+r\Big(1-\frac{1}{\alpha}-\frac{l}{\alpha(n-lr+l)}\Big)( x_{\textup{min}}(G_n) )^{\alpha} \Big)\lambda_{n}.
\end{align*}

$(2)$ Cooper, Desai  and Sahay \cite[Theorem $4.6$]{CDS2024} established that any   hereditary and multiplicative family $\mathcal{F}$ is clonal. Consequently, $\mathrm{Col}(P)$ is clonal. For  $n>n_{1}$, Lemma \ref{clonal} and Claim $2$ give
$$\gamma(G_{n},\mathbf{x})=\frac{\mathbf{x}_{\textup{max}}}{\mathbf{x}_{\textup{min}}}=1+O(n^{-1}).$$
Since  $\mathbf{x}_{\textup{max}}\geq n^{-1/\alpha}$, we have
$$(\mathbf{x}_{\textup{min}})^{\alpha}=\frac{(\mathbf{x}_{\textup{max}})^{\alpha}}{1+O(n^{-1})}= (\mathbf{x}_{\textup{max}})^{\alpha}(1-O(n^{-1}))\geq \frac{1}{n}(1-O(n^{-1})).$$
This shows that there exists a constant $M_{0}>r^{2}$ such that $M=\alpha rM_{0}/(\alpha-1)(r-1)$ and $$(\mathbf{x}_{\textup{min}})^{\alpha}\geq \frac{1}{n}(1-\frac{\hat{\pi}(\mathrm{Col}(P))M}{rn}).$$ Consequently, we also have $$(\mathbf{x}_{\textup{min}}(G_{n}))^{\alpha}\geq \frac{1}{n}(1-\frac{\hat{\pi}(\mathrm{Col}(P))M}{rn}).$$

$(3)$ Define $\varepsilon(n)=M/n$ and $\varepsilon'(n)=\hat{\pi}(\mathrm{Col}(P))M/rn$. Then $0\leq\varepsilon'(n)<\varepsilon(n)\hat{\pi}(\mathrm{Col}(P))/(r-1)$.
 Suppose for contradiction that for some $n> n_{1}$, $\delta(G_{n})<\hat{\pi}(\mathrm{Col}(P))(1-\varepsilon(n))\binom{n}{r-1}$. Note that $\hat{\pi}(\mathrm{Col}(P))>0$ and $(\mathbf{x}_{\textup{min}})^{\alpha}\geq \frac{1}{n}(1-\varepsilon'(n))$. We replicate the
 methodology  of  Lemma \ref{t2} to derive the following analogy of inequality (\ref{qj}):
\begin{displaymath}
\begin{split}
1-\frac{r^{2}}{n}&\leq \Big(\frac{1-\varepsilon(n)}{1-\varepsilon'(n)}\Big)^{\frac{\alpha-1}{\alpha}}
=\Big(1-\frac{\varepsilon(n)-\varepsilon'(n)}{1-\varepsilon'(n)}\Big)^{\frac{\alpha-1}{\alpha}}
\leq \Big(1-(\varepsilon(n)-\varepsilon'(n))\Big)^{\frac{\alpha-1}{\alpha}}.\\
\end{split}
\end{displaymath}
Substituting $\varepsilon(n)-\varepsilon'(n)=\frac{(r-\hat{\pi}(\mathrm{Col}(P)))M}{rn}$ and applying Bernouli's inequality,
\begin{displaymath}
\begin{split}
\big(1-(\varepsilon(n)-\varepsilon'(n))\big)^{\frac{\alpha-1}{\alpha}}&\leq 1-\frac{(\alpha-1)(r-\hat{\pi}(\mathrm{Col}(P)))M}{\alpha rn}\\
&\leq1-\frac{(\alpha-1)(r-1)M}{\alpha rn}\\
&=1-\frac{M_{0}}{n}.
\end{split}
\end{displaymath}
Thus $1-\frac{r^{2}}{n}\leq 1-\frac{M_{0}}{n}$, implying $M_{0}\leq r^{2}$. This contradicts $M_{0}>r^{2}$, completing the proof.
\end{proof}

\subsection{ Spectral extremal hypergraphs in $k$-chromatic $r$-graphs}
In \cite{KNY2014}, Kang, Nikiforov  and Yuan  proposed the following conjecture:

 \begin{conj}
Let $k\geq2$, and let $G$ be a $k$-chromatic $r$-graph of order $n>(r-1)k$.  For every $\alpha\geq1$,
\[
\lambda^{(\alpha)}(G)<\lambda^{(\alpha)}(Q^{r}_{k}(n)),
\]
unless $G=Q^{r}_{k}(n)$.
  \end{conj}
This conjecture has been verified   for    $r=3$ in \cite{KNY2014}. We state that
for $\alpha>1$ and $k\geq2$, the  spectral extremal  $k$-chromatic $r$-graphs asymptotically approach $Q_{k}^{r}(n)$ as $n\to \infty$. Hereafter, the notation $x=y\pm m$ denotes  $y-m\leq x \leq y+m$.
\begin{thm}\label{chro}
Let $\alpha>1$, $k$, $r\geq2$, and let $\mathcal{F}$ be the set of $k$-chromatic $r$-graphs. Let $G_{n}$ be a $k$-chromatic $r$-graph on $n$ vertices with
$\lambda^{(\alpha)}(G_{n})=\lambda^{(\alpha)}(\mathcal{F},n)$, and let $V_{1},\ldots,V_{k}$ be the partition sets of $V(G_{n})$.
Then there exist $n_{0}\in \mathbb{N}$ and $m>0$ such that for all $n\geq n_{0}$, $|V_{i}|= \frac{n}{k}\pm m$ for each $i\in [k]$. In particular,
$\lim\limits_{n\to \infty}\frac{|V_{i}|}{n}=\frac{1}{k}$.
\end{thm}
\begin{proof}
Let  $P=([k],E)$ be an $r$-pattern with $E=[k]^{r}\backslash\cup_{i=1}^{k}\{\{i,\ldots,i\}\}$, where $[k]^{r}$ is the set of all $r$-multisets of elements from $[k]$. Then $\mathrm{Col}(P)=\mathcal{F}$. By a straight calculation, we get  $$\hat{\pi}(\mathrm{Col}(P))=\lim\limits_{n\to \infty}\frac{e(Q_{k}^{r}(n))}{\binom{n}{r}}=1-\frac{1}{k^{r-1}}.$$
By part $(3)$ of Theorem \ref{Co}, there exist $M>0$ and sufficiently large $n_{0}$  such that for all $n\geq n_{0}$,
$$\delta(G_{n})\geq \Big(1-\frac{1}{k^{r-1}}\Big)\Big(1-\frac{M}{n}\Big)\binom{n}{r-1}.$$

Assume that $|V_{i}|=n_{i}$  for $i\in [k]$  and  $n_{1}\geq\cdots\geq n_{k}\geq0$.
From the proof  of Theorem \ref{Co}, we  know that $G_{n}$ must be a complete $k$-chromatic $r$-graph, and so

$$\delta(G_{n})=\binom{n-1}{r-1}-\binom{n_{1}-1}{r-1}\geq \Big(1-\frac{1}{k^{r-1}}\Big)\Big(1-\frac{M}{n}\Big)\binom{n}{r-1}.$$
This implies that
\begin{displaymath}
\begin{split}
\binom{n_{1}-1}{r-1}&\leq\binom{n-1}{r-1}-\Big(1-\frac{1}{k^{r-1}}\Big)\Big(1-\frac{M}{n}\Big)\binom{n}{r-1}\\
&=\frac{1}{(r-1)!}(n^{r-1}-O(n^{r-2}))-\frac{1-1/k^{r-1}}{(r-1)!}(n^{r-1}-O(n^{r-2}))\\
&\leq\frac{n^{r-1}}{k^{r-1}(r-1)!}(1+O(n^{-1})),
\end{split}
\end{displaymath}
 and consequently
$$n_{1}^{r-1}\leq \frac{n^{r-1}}{k^{r-1}}(1+O(n^{-1})).$$
Note that $n_{1}\geq n/k$, we have
$$\frac{n}{k}\leq n_{1}\leq \frac{n}{k}+O(1).$$

On the other hand, we have
$$n=n_{1}+\cdots+n_{k}\leq (k-1)n_{1}+n_{k}\leq \frac{(k-1)n}{k}+O(1)+n_{k},$$
and hence,
$$\frac{n}{k}-O(1)\leq n_{k}\leq\cdots \leq n_{1}\leq \frac{n}{k}+O(1).$$
This completes the proof of Theorem \ref{chro}.
\end{proof}

\section{Spectral Tur\'an-type problems with degree stability}

Building on  Theorem \ref{Co}, certain spectral Tur\'an-type extremal problems admit resolution through degree stability. In subsequent discussions, when addressing spectral Tur\'an-type or classical Tur\'an extremal problems for a family $\mathcal{F}$ of $r$-graphs, we consistently assume that all members of $\mathcal{F}$ contain no isolated vertices.

\begin{defi}[Degree stability\cite{LRM2023}]
Let $\mathcal{F}$ be a  family of $r$-graphs with $\pi(\mathcal{F})>0$, where $r\geq 2$,  and let $\mathfrak{H}$ be a family of $\mathcal{F}$-free $r$-graphs. We say
$\mathcal{F}$ is \emph{degree-stable} with respect  to $\mathfrak{H}$  if  there exist constants $n_{0}$ and $\varepsilon>0$ such that  every $\mathcal{F}$-free $r$-graph $\mathcal{H}$ on $n\geq n_{0}$ vertices with $\delta(\mathcal{H})\geq (\pi(\mathcal{F})/(r-1)!-\varepsilon)n^{r-1}$ is a member of $\mathfrak{H}$.
\end{defi}

The classical Andr\'asfai-Erd\H{o}s-S\'os  Theorem \cite{APS1974} states that for any integer $l\geq2$, the complete graph $K_{l+1}$ is degree-stable with respect to $\mathrm{Col}(K_{l})$.  Tur\'an  problems are closely related to the phenomenon of stability, and usually can be solved    by  stability results of the corresponding graphs or hypergraphs. The stability properties of hypergraphs  have attracted much attention; for details, see \cite{BIJ2017,FS2005,KS2005,LRM2023,MO2006,SL2018}. Now the following result can be viewed as another demonstration of the power of stability phenomenon in spectral Tur\'an-type extremal problems.

\begin{thm}\label{tp}
Let $\alpha>1$ and $r\geq 2$. Let $P$ be an $r$-pattern, and let $\mathcal{F}$ be  a family of $r$-graphs that is degree-stable with respect to  $\mathrm{Col}(P)$. Then for all sufficiently large $n$ and any $n$-vertex $\mathcal{F}$-free $r$-graph $G$, we have \[\lambda^{(\alpha)}(G)\leq\lambda^{(\alpha)}(\mathrm{Col}(P),n),\]
where equality holds only if $G$ is $P$-colorable.
\end{thm}
\begin{proof}
Let $H_{n}\in Mon(\mathcal{F})_{n}$  with $\lambda^{(\alpha)}(H_{n})=\lambda^{(\alpha)}(Mon(\mathcal{F}),n)$, and $G_{n}\in \mathrm{Col}(P)_{n}$  with $\lambda^{(\alpha)}(G_{n})=\lambda^{(\alpha)}(\mathrm{Col}(P),n)$. We claim that for all sufficiently large $n$, $H_{n}\in \mathrm{Col}(P)_{n}$, and hence $\lambda^{(\alpha)}(H_{n})=\lambda^{(\alpha)}(G_{n})$.

Since $\mathcal{F}$ is degree-stable with respect to $\mathrm{Col}(P)$,
we have
$$\pi(\mathcal{F})=\hat{\pi}(Mon(\mathcal{F}))\geq \hat{\pi}(\mathrm{Col}(P))>0.$$
 Therefore,
there exist $n_{0}$ and $\varepsilon \in(0,1)$ such that   every $\mathcal{F}$-free $r$-graph $\mathcal{H}$ on $n\geq n_{0}$ vertices with $\delta(\mathcal{H})\geq (\hat{\pi}(Mon(\mathcal{F}))/(r-1)!-\varepsilon)n^{r-1}$ is contained in $\mathrm{Col}(P)$.
 Choose a constant $\varepsilon'\in (0,\varepsilon\hat{\pi}(Mon(\mathcal{F}))/2(r-1))$.
By Lemma \ref{t4},  there exists a sufficiently large $n_{1}>n_{0}$ such that
 $(x_{\textup{min}}(H_{n_{1}}))^{\alpha}\geq(1-\varepsilon')/n_{1}$.
  Lemma \ref{t2} then yields
 $$\delta(H_{n_{1}})\geq (1-\varepsilon/2)\hat{\pi}(Mon(\mathcal{F}))\binom{n_{1}}{r-1}\geq (\hat{\pi}(Mon(\mathcal{F}))/(r-1)!-\varepsilon)n_{1}^{r-1}.$$ This implies that $H_{n_{1}}\in \mathrm{Col}(P)_{n_{1}}$.

Now, assume that  for some $n\geq n_{1}$, $H_{n}\in \mathrm{Col}(P)_{n}$, and hence
 $\lambda^{(\alpha)}(H_{n})=\lambda^{(\alpha)}(G_{n})$. By
 part $(2)$ of Theorem \ref{Co} (since $n_{1}$ was chose sufficiently large), we have
 $(x_{\textup{min}}(H_{n}))^{\alpha}\geq(1-\varepsilon'/2)/n$.
 Furthermore,  part $(1)$ of Theorem \ref{Co} gives
\begin{align}\label{HG}
\begin{split}
 \lambda^{(\alpha)}(H_{n+1})&\geq\lambda^{(\alpha)}(G_{n+1})\\
 %&\geq\Big(1+r(1-\frac{1}{\alpha}
 %-\frac{l}{\alpha(n-lr+l)})\mu^{\alpha}_{n}\Big)\lambda^{(\alpha)}(G_{n})\\
 &\geq\bigg(1+r\Big(1-\frac{1}{\alpha}
 -\frac{l}{\alpha(n-lr+l)}\Big)(x_{\textup{min}}(H_{n}))^{\alpha}\bigg)\lambda^{(\alpha)}(H_{n})\\
 &\geq\bigg(1+\frac{r}{n}\Big(1-\frac{1}{\alpha}
 -\frac{l}{\alpha(n-lr+l)}\Big)(1-\varepsilon'/2)\bigg)\lambda^{(\alpha)}(H_{n}).
 \end{split}
\end{align}

 Next we  show that
$$(x_{\textup{min}}(H_{n+1}))^{\alpha}\geq \frac{1}{n+1}(1-\varepsilon').$$
Suppose, for contradiction, that
$$(x_{\textup{min}}(H_{n+1}))^{\alpha}<\frac{1}{n+1}(1-\varepsilon').
$$
Then,  there exists
 a principal eigenvector $\mathbf{x}$ of $H_{n+1}$ such that
\begin{equation}\label{i1}
(\mathbf{x}_{\textup{min}})^{\alpha}<\frac{1}{n+1}(1-\varepsilon').
\end{equation}
Applying (\ref{e8}) and (\ref{i1}), Fact $1$, and Bernoulli's inequality, we obtain
\begin{align}\label{i2}
\begin{split}
\frac{\lambda^{(\alpha)}(H_{n})}{\lambda^{(\alpha)}(H_{n+1})}
&\geq\frac{1-r(\mathbf{x}_{\textup{min}})^{\alpha}}{(1-(\mathbf{x}_{\textup{min}})^{\alpha})^{r/\alpha}}\\
&\geq\Big(1-\frac{r(1-\varepsilon')}{n+1}\Big)\Big(1-\frac{1-\varepsilon'}{n+1}\Big)^{-r/\alpha}\\
&\geq\Big(1-\frac{r(1-\varepsilon')}{n+1}\Big)\Big(1+\frac{r(1-\varepsilon')}{\alpha(n+1)}\Big)\\
&= 1-\frac{r(1-1/\alpha)(1-\varepsilon')}{n+1}
-\frac{r^{2}(1-\varepsilon')^{2}}{\alpha(n+1)^{2}}\\
&= 1-\frac{r(1-1/\alpha)(1-\varepsilon')}{n+1}
-o(n^{-1}).
\end{split}
\end{align}
From (\ref{HG}),
\begin{align}\label{i3}
\begin{split}
\frac{\lambda^{(\alpha)}(H_{n})}{\lambda^{(\alpha)}(H_{n+1})}
&\leq \frac{1}{1+r\big(1-1/\alpha-l/\alpha(n-lr+l)\big)\big(1-\varepsilon'/2\big)n^{-1}}\\
&= 1-r\big(1-1/\alpha-l/\alpha(n-lr+l)\big)\big(1-\varepsilon'/2\big)n^{-1}+o(n^{-1})\\
&= 1-r(1-1/\alpha)(1-\varepsilon'/2)n^{-1}+o(n^{-1})\\
&\leq 1-r(1-1/\alpha)(1-\varepsilon'/2)(n+1)^{-1}+o(n^{-1}).
\end{split}
\end{align}
Combining (\ref{i2}) and (\ref{i3}) gives
$$ \frac{r(1-1/\alpha)(1-\varepsilon'/2)}{n+1}
  \leq\frac{r(1-1/\alpha)(1-\varepsilon')}{n+1}
+o(n^{-1}),$$
which implies
$$\frac{r(1-1/\alpha)\varepsilon'}{2(n+1)}<o(n^{-1}).$$
Consequently,
$$\frac{r(1-1/\alpha)\varepsilon'}{2}<o(1).$$
This is a contradiction. Hence,
$$(x_{\textup{min}}(H_{n+1}))^{\alpha}\geq \frac{1}{n+1}(1-\varepsilon').$$

By Lemma \ref{t2} again,  we deduce that
 $$\delta(H_{n+1})\geq (1-\varepsilon/2)\hat{\pi}(Mon(\mathcal{F}))\binom{n+1}{r-1}\geq (\hat{\pi}(Mon(\mathcal{F}))/(r-1)!-\varepsilon)(n+1)^{r-1}.$$
 Thus $H_{n+1}\in \mathrm{Col}(P)_{n+1}$.
By induction on $n$, we prove that $H_{n}\in \mathrm{Col}(P)_{n}$ for all  $n\geq n_{1}$,
and so $\lambda^{(\alpha)}(H_{n})=\lambda^{(\alpha)}(G_{n})$.
This shows that $\lambda^{(\alpha)}(G)\leq\lambda^{(\alpha)}(H_{n})=\lambda^{(\alpha)}(\mathrm{Col}(P),n)$ for all  $n\geq n_{1}$,
completing the proof.
\end{proof}

%\noindent {\bf Remark.} The result above extends Theorem $1.2$ from \cite{HLZ2024}, which was derived using a general theorem established by Keevash, Lenz  and Mubayi \cite[Theorem $1.4$]{KLM2014}.
\noindent {\bf Remark.}
Theorem \ref{tp} provides a more general and versatile reduction framework compared to Theorem \ref{oldcri}. In particular, it eliminates the need to verify the technical spectral increment condition (\ref{oldcri1}), replacing it with the more combinatorial concept of degree-stability.

\medskip{By a similar  method,    Theorem \ref{tp} can be generalized to the case  for multiple $r$-patterns  as follows.}

\begin{thm}\label{mulp}
Let  $\alpha>1$ and $r\geq 2$. Let $P_{i}$ be $r$-patterns for $i\in[t]$, and  $\mathcal{F}$ be a family of $r$-graphs that is degree-stable with respect to  $\cup_{i=1}^{t}\mathrm{Col}(P_{i})$. Then for all sufficiently large  $n$ and any $n$-vertex $\mathcal{F}$-free $r$-graph   $G$,   we have $\lambda^{(\alpha)}(G)\leq \max_{ i\in [t]}\lambda^{(\alpha)}(\mathrm{Col}(P_{i}),n)$. Moreover, if equality holds then $G$ is $P_{i}$-colorable for some $i\in[t]$.
\end{thm}

The following corollary is an immediate consequence of Lemma  \ref{clonal}, Theorems \ref{Co} and \ref{tp}.

\begin{cor}
Let $\alpha>1$ and $r\geq2$. Let $P$ be an  $r$-pattern, and $\mathcal{F}$ be  a family of $r$-graphs that is degree-stable with respect to  $\mathrm{Col}(P)$. Suppose that $G$ is an $\mathcal{F}$-free $r$-graph  on $n$ vertices satisfying $\lambda^{(\alpha)}(G)=\lambda^{(\alpha)}(Mon(\mathcal{F}),n)$,
and $\mathbf{x}$ is a principle eigenvector of $G$. Then for  sufficiently large $n$,
$$\gamma(G,\mathbf{x})=1+O(n^{-1}).$$
\end{cor}

\begin{cor}\label{bound}
Let $\alpha\geq1$ and $r\geq 2$. Let $P$ be an $r$-pattern, and  $\mathcal{F}$ be a family of $r$-graphs that is degree-stable with respect to $\mathrm{Col}(P)$. Then for all   sufficiently large $n$ and any $\mathcal{F}$-free $n$-vertex  $r$-graph   $G$, we have $\lambda^{(\alpha)}(G)\leq\hat{\pi}(\mathrm{Col}(P))n^{r-r/\alpha}$.
Moreover, for $\alpha>1$,
$$\lambda^{(\alpha)}(Mon(\mathcal{F}),n)=\hat{\pi}(\mathrm{Col}(P))n^{r-r/\alpha}-|O(1)|n^{r-r/\alpha-1}.$$
\end{cor}
\begin{proof}
 Theorems \ref{t8} and \ref{tp} imply that for $\alpha>1$,
\begin{equation}\label{sr-edgedensity}
\lambda^{(\alpha)}(G)\leq\hat{\pi}(\mathrm{Col}(P))n^{r-r/\alpha}.
\end{equation}
Applying  Lemma \ref{conti} and  continuity
 arguments to \eqref{sr-edgedensity}, we have
$$\lambda^{(1)}(G)=\lim\limits_{\alpha\to 1^{+}}\lambda^{(\alpha)}(G)\leq \lim\limits_{\alpha\to 1^{+}}\hat{\pi}(\mathrm{Col}(P))n^{r-r/\alpha}=\hat{\pi}(\mathrm{Col}(P)).$$
Thus the first statement holds true.

Furthermore, combining Lemmas \ref{exi} and \ref{t1} yields  for $\alpha>1$,
\begin{align*}
\begin{split}
\lambda^{(\alpha)}(Mon(\mathcal{F}), n)&\geq\frac{\lambda^{(\alpha)}(\mathrm{Col}(P)){(n)_{r}}}{n^{r/\alpha}}
=\frac{\hat{\pi}(\mathrm{Col}(P)){(n)_{r}}}{n^{r/\alpha}}\\
&\geq \hat{\pi}(\mathrm{Col}(P))n^{r-r/\alpha}-|O(1)|n^{r-r/\alpha-1}.
\end{split}
\end{align*}
Thus the second statement follows from the above result and \eqref{sr-edgedensity}.
\end{proof}

\begin{lem}[\hspace{1sp}\cite{KNY2014}]\label{kpt}
Let $l\geq r\geq2$, and let $G$ be an $l$-partite $r$-graph of order $n$. For all $\alpha>1$,
$$\lambda^{(\alpha)}(G)< \lambda^{(\alpha)}(T^{r}_{l}(n)),$$
unless $G=T^{r}_{l}(n)$.
\end{lem}

Next, we characterize  spectral extremal hypergraphs among all $\mathcal{F}$-free $r$-graphs, where  $\mathcal{F}$ is degree-stable with respect to $\mathrm{Col}(K^{r}_{l})$.  This result generalizes Theorem $3.3$ in \cite{ZLF2024}.

\begin{cor}\label{trl}
Let  $\alpha\geq1$ and $l\geq r\geq 2$.  Let  $\mathcal{F}$ be a family of $r$-graphs that is degree-stable with respect to $\mathrm{Col}(K^{r}_{l})$.
 Then, for all  sufficiently large $n$ and any $n$-vertex $\mathcal{F}$-free  $r$-graph $G$,
 we have $\lambda^{(\alpha)}(G)\leq\lambda^{(\alpha)}(T^{r}_{l}(n))$. The equality holds  if and only if  $G=T^{r}_{l}(n)$ for $\alpha>1$, and if  $ G\supseteq K^{r}_{l}$ for $\alpha=1$.
 %In particular, $\lambda^{(1)}(Mon(\mathcal{F},n))$
\end{cor}
\begin{proof}
 Note that  $\mathrm{Col}(K^{r}_{l})$ consists of all $l$-partite $r$-graphs and satisfies $\hat{\pi}(\mathrm{Col}(K^{r}_{l}))=(l)_{r}/l^{r}$.
For $\alpha>1$, the conclusion follows directly from  Theorem \ref{tp} and Lemma \ref{kpt}. Consider now the case $\alpha=1$.
By Corollary \ref{bound}, we have
$\lambda^{(1)}(G)\leq (l)_{r}/l^{r}$.
If  $G$ contains a copy of $K^{r}_{l}$, then $$\lambda^{(1)}(G)\geq \lambda^{(1)}(K^{r}_{l})\geq P_{K^{r}_{l}}((1/l,\ldots,1/l))=(l)_{r}/l^{r},$$ which implies $\lambda^{(1)}(G)=(l)_{r}/l^{r}$. Therefore, $\lambda^{(1)}(G)\leq \lambda^{(1)}(T^{r}_{l}(n))=(l)_{r}/l^{r}$, completing the proof.
\end{proof}

\medskip{ A graph $H$ with $\chi(H)=l$ is called  \emph{color-critical} if  it contains an edge $e$ such that $\chi(H-e)=l-1$.  As shown in Table $1$ of \cite{HLZ2024}, the $r$-expansion of any such color-critical graph $F$ (with $\chi(F)=l+1$)  is degree-stable with respect to $\mathrm{Col}(K^{r}_{l})$. Consequently, by Corollary \ref{trl}, we obtain  the following result.}

\begin{thm}\label{ccg}
Let $\alpha\geq1$, $l\geq r\geq 2$ , and $F$ be a color-critical graph with $\chi(F)=l+1$. Then
there exists $n_{0}$, such that for any $F^{(r)}$-free $r$-graph $G$ on $n>n_{0}$ vertices, $\lambda^{(\alpha)}(G)\leq\lambda^{(\alpha)}(T^{r}_{l}(n))$. The equality holds  if and only if  $G=T^{r}_{l}(n)$ for $\alpha>1$, and if  $ G\supseteq K^{r}_{l}$ for $\alpha=1$.
\end{thm}

%\noindent {\bf Remark.} Using Corollary \ref{trl} and \ref{b4n}, we can more succinctly obtain the spectral extremal results presented in Section $4$  of \cite{ZLF2024}.

\section{Application of spectral method to  classical Tur\'an  problems}

As noted  by Nikiforov \cite{N2014B}, the $\alpha$-spectral radius provides a unified framework  incorporating the Lagrangian, the number of edges, the spectral radius, and related concepts. This enables a cohesive approach to addressing these problems. We demonstrate this principle by resolving
classical edge-Tur\'an extremal problems for specified families of $r$-graphs using the $\alpha$-spectral radius, thereby providing further validation of this idea.

\begin{thm}\label{colore}
    Let $P$ be an $r$-pattern with $r \ge 2$, and let $\mathcal{F}$ be a family of $r$-graphs that is degree-stable with respect to $\operatorname{Col}(P)$. Then for all sufficiently large $n$ and any $n$-vertex $\mathcal{F}$-free $r$-graph $G$, we have
    \[
        e(G) \le ex(\operatorname{Col}(P), n).
    \]
    Moreover, if equality holds, then $G$ is $P$-colorable.
\end{thm}

\begin{proof}
    By Theorem~\ref{tp}, for any $\alpha > 1$ and sufficiently large $n$,
    \[
        \lambda^{(\alpha)}(G) \le \lambda^{(\alpha)}(\operatorname{Col}(P), n) = \max_{H \in \operatorname{Col}(P)_n} \lambda^{(\alpha)}(H).
    \]
    Hence
    \[
        \begin{split}
            r!\, e(G) &= \lim_{\alpha \to \infty} \lambda^{(\alpha)}(G)
                        \le \varlimsup_{\alpha \to \infty} \max_{H \in \operatorname{Col}(P)_n} \lambda^{(\alpha)}(H) \\
                      &= \max_{H \in \operatorname{Col}(P)_n} \varlimsup_{\alpha \to \infty} \lambda^{(\alpha)}(H)
                       = \max_{H \in \operatorname{Col}(P)_n} r!\, e(H),
        \end{split}
    \]
    which yields
    \begin{equation}\label{ecol-1}
        e(G) \le ex(\operatorname{Col}(P), n).
    \end{equation}

    Now assume equality holds in \eqref{ecol-1}. Choose a vertex $v \in V(G)$ with $d_G(v)=\delta(G)$. Then for sufficiently large $n$,
    \[
        \begin{split}
            \delta(G) &= e(G) - e(G-v) \\
                      &= ex(\operatorname{Col}(P), n) - e(G-v) \\
                      &\ge ex(\operatorname{Col}(P), n) -ex(\operatorname{Col}(P), n-1).
        \end{split}
    \]

    Next we estimate a lower bound for $ex(\operatorname{Col}(P), n) - ex(\operatorname{Col}(P), n-1)$.
    Let $H \in \operatorname{Col}(P)$ be an $(n-1)$-vertex $r$-graph with $e(H) = ex(\operatorname{Col}(P), n-1)$, and let $\Delta(H)$ be its maximum degree. Since the sequence $ex(\operatorname{Col}(P), m)\big/\binom{m}{r}$ is decreasing and converges to $\hat{\pi}(\operatorname{Col}(P))$, we have
    \[
        \begin{split}
            \Delta(H) &\ge \frac{r \cdot ex(\operatorname{Col}(P), n-1)}{n-1} \\
                      &\ge \frac{r \hat{\pi}(\operatorname{Col}(P)) \binom{n-1}{r}}{n-1} \\
                      &= \hat{\pi}(\operatorname{Col}(P)) \binom{n-2}{r-1}.
        \end{split}
    \]

    Take $w \in V(H)$ with $d_H(w)=\Delta(H)$. Construct an $n$-vertex $r$-graph $H'$ from $H$ by adding a new vertex $w'$ and, for each edge $e \in E_H(w)$, adding the edge $e + w' - w$. Clearly $H'$ is also $P$-colorable. Therefore
    \[
        \begin{split}
            \Delta(H) = d_{H'}(w') &= e(H') - e(H) \\
                                   &= e(H') - ex(\operatorname{Col}(P), n-1) \\
                                   &\le ex(\operatorname{Col}(P), n) - ex(\operatorname{Col}(P), n-1).
        \end{split}
    \]
    Consequently,
    \[
        \delta(G) \ge ex(\operatorname{Col}(P), n) - ex(\operatorname{Col}(P), n-1)
                   \ge \hat{\pi}(\operatorname{Col}(P)) \binom{n-2}{r-1}.
    \]

    Because $\mathcal{F}$ is degree-stable with respect to $\operatorname{Col}(P)$, there exist $n_0$ and $\varepsilon \in (0,1)$ such that every $\mathcal{F}$-free $r$-graph $\mathcal{H}$ on $n \ge n_0$ vertices with
    \[
        \delta(\mathcal{H}) \ge \Bigl(\frac{\pi(\mathcal{F})}{(r-1)!} - \varepsilon\Bigr) n^{r-1}
    \]
    belongs to $\operatorname{Col}(P)$. Noting that $\pi(\mathcal{F}) = \hat{\pi}(\operatorname{Col}(P))$, we obtain for sufficiently large $n$ that
    \[
        \delta(G) \ge \hat{\pi}(\operatorname{Col}(P)) \binom{n-2}{r-1}
                   \ge \Bigl(\frac{\hat{\pi}(\operatorname{Col}(P))}{(r-1)!} - \varepsilon\Bigr) n^{r-1}.
    \]
    Hence $G \in \operatorname{Col}(P)$, which completes the proof.
\end{proof}

The following generalization follows similarly, and its proof is omitted.

\begin{thm}
Let $P_{i}$ be $r$-patterns for $i\in[t]$, where $r\geq 2$.  Let $\mathcal{F}$ be a family of $r$-graphs that is degree-stable with respect to  $\cup_{i=1}^{t}\mathrm{Col}(P_{i})$. Then for all sufficiently large $n$ and any $n$-vertex $\mathcal{F}$-free  $r$-graph $G$,   we have $e(G)\leq \max_{ i\in [t]}ex(\mathrm{Col}(P_{i}),n)$.
\end{thm}

By Theorem \ref{colore}, we obtain the following result.

\begin{cor}\label{corr-2}
Let  $l\geq r\geq 2$, and  $\mathcal{F}$ be a family of $r$-graphs that is degree-stable with respect to $\mathrm{Col}(K^{r}_{l})$.
Then, for all  sufficiently large $n$ and any $n$-vertex $\mathcal{F}$-free  $r$-graph $G$,
we have $e(G)\leq e(T^{r}_{l}(n))$. The equality holds  if and only if  $G=T^{r}_{l}(n)$
%In particular, $\lambda^{(1)}(Mon(\mathcal{F},n))$
\end{cor}

 In \cite{MV2016}, the authors note  that Alon and Pikhurko obtained the Tur\'an number of the $r$-expansion of any
 color-critical graph, and this result  can also be derived from Corollary \ref{corr-2}

\begin{cor}
Let $l\geq r\geq 2$ , and $F$ be a color-critical graph with $\chi(F)=l+1$. Then, for all  sufficiently large $n$ and any $n$-vertex $\mathcal{F}$-free  $r$-graph $G$,
we have $e(G)\leq e(T^{r}_{l}(n))$. The equality holds  if and only if  $G=T^{r}_{l}(n)$
\end{cor}

\medskip{The following result, corresponding to Corollary \ref{bound},   follows directly from Theorems  \ref{t8} and \ref{colore}.}

\begin{cor}\label{ebound}
Let $P$ be an $r$-pattern, $r\geq 2$.  Let $\mathcal{F}$ be a family of $r$-graphs that is degree-stable with respect to $\mathrm{Col}(P)$.  Then for all sufficiently large $n$ and any $n$-vertex $\mathcal{F}$-free  $r$-graph $G$,  we have $e(G)\leq\hat{\pi}(\mathrm{Col}(P))n^{r}/r!$.
\end{cor}

Let $\mathcal{P}$ be a hereditary  property of $r$-graphs. Denote by $EX(\mathcal{P},n)$ (resp. $SPEX_{\alpha}(\mathcal{P},n)$)  the set of $n$-vertex $r$-graphs in $\mathcal{P}_{n}$ achieving the maximum size (resp. maximum $\alpha$-spectral radius) . A natural problem is to investigate the relation between $SPEX_{\alpha}(\mathcal{P},n)$ and $EX(\mathcal{P},n)$. Based on the symmetric structure of the $P$-colorable $r$-graphs, we propose the following question.

\begin{pro}\label{p}
Let $\alpha>1$ and $r\geq2$. For all $r$-patterns $P$, is it true that for sufficiently large $n$,
$$SPEX_{\alpha}(\mathrm{Col}(P),n)\subseteq EX(\mathrm{Col}(P),n)?$$
\end{pro}

This problem is rather difficult in general, but we can give an affirmative answer under certain special conditions.

%\begin{thm}\label{2pa}
%Let $\alpha>1$. Then for all $2$-patterns $P$, we have
%$$SPEX_{\alpha}(\mathrm{Col}(P),n)=EX(\mathrm{Col}(P),n).$$
%\end{thm}
%\begin{proof}
%$\mathbf{case 1.}$
%Given a $2$-pattern $P=([l],E)$.
%If there exists  $i\in [l]$ such that $\{i,i\}\in E$, then $\mathrm{Col}(P)=\mathrm{Col}(\{i\},\{\{i,i\}\})$, which is the set of all $2$-graphs. Consequently,
%$$SPEX_{\alpha}(\mathrm{Col}(P),n)=EX(\mathrm{Col}(P),n)=\{K_{n}\}.$$

%In the following, we assume that $\{i,i\}\notin E$ for any $i\in [l]$. Then, the  $2$-pattern $P$ is a simple graph with edge set $E$. It follows that a $P$-colorable graph is a blow-up of $P$.
%Without loss of generality, let $K_{t}$ be the largest clique in $P$, where $t\geq2$, then $P$ is a $t$-partite graph. This implies that $\mathrm{Col}(P)=\mathrm{Col}(K_{t})$. By Lemma \ref{kpt}, we have
%$$SPEX_{\alpha}(\mathrm{Col}(P),n)=EX(\mathrm{Col}(K_{t}),n)=\{T_{t-1}(n)\},$$
%and hence, complete the proof of Theorem \ref{2pa}.
%\end{proof}

\begin{thm}\label{hmhm}
Let $\alpha\geq1$ and $r\geq2$. Let $\mathcal{P}$ be a flat property of $r$-graphs. If
%for some $n$ satisfying
$ex(\mathcal{P},n)=\hat{\pi}(\mathcal{P})n^{r}/r!$, then
$$SPEX_{\alpha}(\mathcal{P},n)= EX(\mathcal{P},n).$$
\end{thm}
\begin{proof}
For any $G\in EX(\mathcal{P},n)$, by inequality (\ref{zhy}), we have
$$\lambda^{(\alpha)}(G)\geq r!e(G)/n^{r/\alpha}=\hat{\pi}(\mathcal{P})n^{r-r/\alpha}.$$
On the other hand, Theorem \ref{t8} yields  $\lambda^{(\alpha)}(G)\leq\hat{\pi}(\mathcal{P})n^{r-r/\alpha}$. Consequently, $$\lambda^{(\alpha)}(G)=\lambda^{(\alpha)}(\mathcal{P},n)=\hat{\pi}(\mathcal{P})n^{r-r/\alpha}.$$
Therefore, $EX(\mathcal{P},n)\subseteq SPEX_{\alpha}(\mathcal{P},n)$.

For any $H\in SPEX_{\alpha}(\mathcal{P},n)$,
we have $\lambda^{(\alpha)}(H)=\lambda^{(\alpha)}(\mathcal{P},n)=\hat{\pi}(\mathcal{P})n^{r-r/\alpha}$.
Applying Lemma \ref{ttt6}, we obtain
$$\hat{\pi}(\mathcal{P})^{1/\alpha}(r!e(H))^{1-1/\alpha}\geq \hat{\pi}(\mathcal{P})n^{r-r/\alpha},$$
which implies $e(H)\geq \hat{\pi}(\mathcal{P})n^{r}/r!$. This shows that
 $SPEX_{\alpha}(\mathcal{P},n)\subseteq EX(\mathcal{P},n)$.
\end{proof}

\noindent {\bf Remark.} The result above generalizes Theorem  6.1 of \cite{LNWK2024}. This extension follows from the fundamental observation that  an $r$-graph $H$ is $\mathcal{M}$-hom-free (as defined in \cite{LNWK2024}) if and only if $H(\mathbf{t})$ is $\mathcal{M}$-free for every $\mathbf{t}\in \mathbb{Z}^{\nu(H)}$, which implies that $Mon(\mathcal{M})$ is multiplicative. Further details are omitted here.

\begin{cor}
Let $\alpha\geq1$, and let $P$ be an $r$-pattern. If
%for some $n$ satisfying
$ex(\mathrm{Col}(P),n)=\hat{\pi}(\mathrm{Col}(P))n^{r}/r!$, then
$$SPEX_{\alpha}(\mathrm{Col}(P),n)= EX(\mathrm{Col}(P),n).$$
\end{cor}

 Applying Theorem \ref{tp} yields  the following result directly.

\begin{prop}
Let $\alpha>1$, $r\geq 2$, and let $P$ be an  $r$-pattern. Let $\mathcal{F}$ be a family of $r$-graphs that is degree-stable with respect to  $\mathrm{Col}(P)$. If the answer to Problem \ref{p} is affirmative, then
$$SPEX_{\alpha}(Mon(\mathcal{F}),n)\subseteq EX(Mon(\mathcal{F}),n)$$
for sufficiently large $n$.
\end{prop}

\section*{Data availability}
 No data was used for the research described in the article.

\appendix

\section*{Appendix}

\begin{proof}[{\bf Proof of Lemma \ref{exi}}]
For every integer $n> r$, define $\lambda^{(\alpha)}_{n}:=\lambda^{(\alpha)}(\mathcal{P},n)$. Let $G\in \mathcal{P}_{n}$
be an $r$-graph satisfying $\lambda^{(\alpha)}(G)=\lambda^{(\alpha)}(\mathcal{P},n)$, and let  $\mathbf{x}=(x_1,\ldots,x_n)$ be a principal eigenvector for $\lambda^{(\alpha)}(G)$ with $\|\mathbf{x}\|_{\alpha}=1$.

For $\alpha=1$,  by assumption $(b)$, we have $\lambda^{(1)}_{n}\geq \lambda^{(1)}_{n-1}$.
Moreover, Maclaurin's inequality implies
$$\lambda_{n}^{(1)}=P_{G}(\mathbf{x})\leq r!\sum_{1\leq i_{1}<\cdots<i_{r} \leq n}x_{i_{1}}\cdots x_{i_{r}}
\leq (x_{1}+\cdots+x_{n})^{r}=1.$$
Thus, the sequence $\Big\{\lambda_{n}^{(1)}\Big\}_{n=1}^{\infty}$  converges to a limit $\lambda$,  and we conclude
$$\lambda=\lim\limits_{\alpha\to \infty}\lambda_{n}^{(1)}n^{r-r}=\lambda^{(1)}(\mathcal{P}).$$

For $\alpha>1$, noting that $(\mathbf{x}_{\textup{min}})^{\alpha}\leq 1/n$,
by (\ref{e8}) and Fact $1$, we have
\begin{equation*}
 \frac{\lambda^{(\alpha)}_{n-1}}{\lambda^{(\alpha)}_{n}}
\geq\frac{1-r(\mathbf{x}_{\textup{min}})^{\alpha}}{(1-(\mathbf{x}_{\textup{min}})^{\alpha})^{r/\alpha}}\geq
   \frac{1-r/n}{(1-1/n)^{r/\alpha}}.
\end{equation*}
This implies that
$$\frac{\lambda^{(\alpha)}_{n-1}(n-1)^{r/\alpha}}{(n-1)_{r}}\geq \frac{\lambda^{(\alpha)}_{n}n^{r/\alpha}}{(n)_{r}}.$$
Therefore, the sequence $\Big\{\frac{\lambda^{(\alpha)}_{n}n^{r/\alpha}}{(n)_{r}}\Big\}_{n=1}^{\infty}$ is decreasing and hence convergent. This completes the proof.
\end{proof}

\begin{proof}[{\bf Proof of Theorem \ref{t8}}]
%Taking an $n$-vector $\mathbf{x}=(1/n,\ldots,1/n)$, we see that $\|\mathbf{x}\|_{1}=1$. By Lemma \ref{exi}, we have
%$$\frac{r!e(G)}{n^{r}}=P_{G}(\mathbf{x})\leq \lambda^{(1)}(G)\leq \lambda^{(1)}(\mathcal{P},n)\leq\lambda^{(1)}(\mathcal{P})=\pi(\mathcal{P}).$$
%Therefore, $e(G)\leq\pi(\mathcal{P})n^{r}/r!$.
Let $\mathbf{x}=(x_1,\ldots,x_n)$ be a principal eigenvector for $\lambda^{(\alpha)}(G)$ with $\|\mathbf{x}\|_{\alpha}=1$. Then
$$\frac{\lambda^{(\alpha)}(G)}{n^{r-r/\alpha}}\leq \frac{P_{G}(\mathbf{x})}{n^{r-r/\alpha}}
=P_{G}((x_1/n^{1-1/\alpha},\ldots,x_n/n^{1-1/\alpha})).$$
By Power-Mean inequality, for any $\alpha\geq1$,
$$\frac{x_1+\cdots+x_n}{n^{1-1/\alpha}}\leq(x_1^\alpha+\cdots+x_n^\alpha)^{1/\alpha}=1.$$
From  Lemma \ref{exi} it follows that
$$\frac{\lambda^{(\alpha)}(G)}{n^{r-r/\alpha}}\leq \lambda^{(1)}(G)\leq \lambda^{(1)}(\mathcal{P},n)\leq\lambda^{(1)}(\mathcal{P})=\pi(\mathcal{P}).$$
Since $\lambda^{(\alpha)}(G)\geq r!e(G)/n^{r/\alpha}$, we conclude
$$e(G)\leq\pi(\mathcal{P})n^{r}/r!,$$
completing the proof.
\end{proof}

\noindent {\bf Fact A.}  Let $P=([l],E)$ be an $r$-pattern. Then $\mathrm{Col}(P)$ is hereditary and multiplicative.
\begin{proof}
We first show that $\mathrm{Col}(P)$ is hereditary. Let $H \in \mathrm{Col}(P)$ and let $H'$ be an induced subgraph of $H$. There exists a homomorphism $\phi: V(H) \to [l]$ such that for every edge $\{v_1, \dots, v_r\} \in E(H)$, the multiset $\{\phi(v_1), \dots, \phi(v_r)\} \in E$. Restrict $\phi$ to $V(H')$. For any edge $\{v_1, \dots, v_r\} \in E(H')$, since $H'$ is induced, this edge is also in $E(H)$, so $\{\phi(v_1), \dots, \phi(v_r)\} \in E$. Thus, $\phi|_{V(H')}$ is a homomorphism from $H'$ to $P$, so $H' \in \mathrm{Col}(P)$.

 Next, we prove that $\mathrm{Col}(P)$ is multiplicative. Let $H \in \mathrm{Col}(P)_n$ and let $\mathbf{t} \in \mathbb{Z}^n_{>0}$. Consider the blow-up $H(\mathbf{t})$. There exists a homomorphism $\phi: V(H) \to [l]$. Define $\psi: V(H(\mathbf{t})) \to [l]$ by mapping every vertex in the blow-up of $v \in V(H)$ to $\phi(v)$. For any edge $\{w_1, \dots, w_r\} \in E(H(\mathbf{t}))$, each $w_i$ lies in the blow-up of some $v_i \in V(H)$, and $\{v_1, \dots, v_r\}$ is an edge in $H$ (by definition of blow-up). Then $\{\psi(w_1), \dots, \psi(w_r)\} = \{\phi(v_1), \dots, \phi(v_r)\} \in E$. Thus, $\psi$ is a homomorphism from $H(\mathbf{t})$ to $P$, so $H(\mathbf{t}) \in \mathrm{Col}(P)$.

Hence, $\mathrm{Col}(P)$ is hereditary and multiplicative.
\end{proof}

\end{document}